\documentclass[12pt]{amsart}
\usepackage{amsmath,amssymb,amsfonts,amsthm,latexsym,amstext,amscd}
\usepackage{epsfig,graphics,graphicx,color}
\usepackage{epstopdf}
\usepackage{pdfsync}
\usepackage{enumerate}
\usepackage[applemac]{inputenc}
\usepackage{color}
\usepackage{hyperref}\hypersetup{pdfborder={0 0 0},colorlinks,citecolor=red,filecolor=red,linkcolor=red,urlcolor=green}
\usepackage[all]{xy}
\usepackage{wasysym}
\makeatletter \thm@headfont{\bfseries\scshape} \makeatother
\allowdisplaybreaks

\newtheorem{thm}{Theorem}
\newtheorem{lem}[thm]{Lemma}
\newtheorem{cor}[thm]{Corollary}
\newtheorem{pro}[thm]{Proposition}
\newtheorem{exmp}[thm]{Example}
\newtheorem{defn}[thm]{Definition}
\newtheorem{rem}[thm]{Remark}

\begin{document}
\title[Sofic measures and densities of level sets]
      {Sofic measures and densities of level sets}

\address{{\bf{Alain THOMAS}}\\
448 all\'ee des Cantons, 83640 Plan d'Aups Sainte Baume}
\email{\rm alain-yves.thomas@laposte.net}

\author[A.~THOMAS]
       {Alain~THOMAS}



\begin{abstract}
The Bernoulli convolution associated to the real $\beta>1$ and the  probability vector $(p_0,\dots,p_{d-1})$ is a probability measure $\eta_{\beta,p}$ on $\mathbb R$, solution of the self-similarity relation $\displaystyle\eta=\sum_{k=0}^{d-1}p_k\cdot\eta\circ S_k$ where $S_k(x)=\frac{x+k}\beta$. If $\beta$ is an integer or a Pisot algebraic number with finite Rényi expansion, $\eta_{\beta,p}$ is sofic and a Markov chain is naturally associated. If $\beta=b\in\mathbb N$ and $p_0=\dots=p_{d-1}=\frac1d$, the study of $\eta_{b,p}$ is close to the study of the order of growth of the number of representations in base $b$ with digits in $\{0,1,\dots,d-1\}$. In the case $b=2$ and $d=3$ it has also something to do with the metric properties of the continued fractions.
\end{abstract}

\maketitle
\addtocounter{section}{-1}
\parindent0pt

\section{Introduction}

The different sections of this paper are relatively independent. A sofic probability measure on a space $\{0,1,\dots,b-1\}^{\mathbb N}$, i.e. the image of a Markov probability measure by a shift-commuting continuous map, is representable by products of matrices as explained in Theorem \ref{Msofic}. The measures defined by Bernoulli convolution \cite{PSS}, i.e.
$$
\eta_{\beta,p}:=\begin{matrix}{\scriptstyle\infty}\\{\Large\hexstar}\\{\scriptstyle n=1}\end{matrix}\left(\sum_{k=0}^{d-1}p_k\delta_{\frac k{\beta^n}}\right)
$$
where $\beta\in\mathbb R$, $p_0,\dots,p_{d-1}>0$ and $\sum_kp_k=1$, are sofic if $\beta$ is an integer. They are also sofic when $\beta$ is a Pisot number (i.e., an algebraic number whose conjugates belong to the open unit disk) with finite Rényi expansion \cite{R}. A transducer of normalization \cite[Theorem~2.3.39]{FS} is naturally associated to $\beta$ and~$d$. The matrices associated to the measure $\eta_{\beta,p}$ are easy to define when $\beta$ is a integer, $\beta=b\ge2$. In the case $p_0=\dots=p_{d-1}=\frac1d$ they are used by different authors (see for instance \cite{Pr}) because they are related to the number of representations of the integer $n$ in base~$b$ with digits in $\{0,\dots,d-1\}$, let $\mathcal N(n)$; and the Hausdorff dimensions of the level sets of $\eta_{b,p}$ are related with the lower exponential densities of some sets of integers, on which $\mathcal N(n)$ has a given order of growth (Theorem \ref{levels}). There exist many partial results about the level sets (see for instance \cite{F1} and \cite{F2}) of the linearly representable measures (see \cite{BP} and \cite{BPW}).

\section{Sofic subshifts}

By subshifts we mean closed subsets of $\{0,1,\dots,b-1\}^{\mathbb N}$ invariant by the shift $\sigma:(\omega_n)_{n\in\mathbb N}\mapsto(\omega_{n+1})_{n\in\mathbb N}$. Let us give two equivalent definitions of the sofic subshifts.

\begin{defn}\label{la1}
A sofic subshift is a subshift recognizable by a finite automaton \cite{B}.
\end{defn}

\begin{defn}\label{thesecond}
A subshift is a sofic iff it is the image of a topological Markov subshift by a letter-to-letter morphism.
\end{defn}

In this definition, "letter-to-letter morphism" can be replaced by "continuous morphism", because any continuous morphism (i.e. any continuous map, for the usual topology, commuting with the shift) has the form
$$
\varphi\big((\omega_n)_{n\in\mathbb N}\big)=\big(\psi(\omega_{n-s}\omega_{n-s+1}\dots\omega_{n+s})\big)_{n\in\mathbb N}
$$
with $\psi:\{0,1,\dots,b-1\}^{2s+1}\to\{0,1,\dots,b'-1\}$, being understood that $\omega_n=0$ for $n\le0$.

\begin{exmp}\label{exmp1}
A example of sofic subshift in the sense of Definition \ref{la1}, is the set of the labels of the infinite paths in the following automaton: 

\includegraphics[trim = 0 220 0 220,clip,scale=0.4]{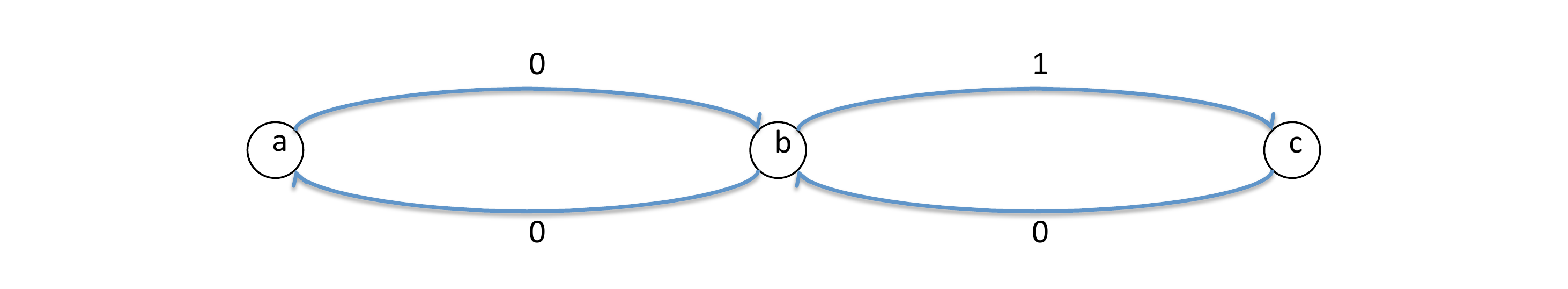}

(this subshift is the set of the sequences $(\omega_n)_{n\in\mathbb N}\in\{0,1\}^\mathbb N$ without factor $10^{2i}1$ for any $i\in\mathbb N\cup\{0\}$, hence it excludes an infinite set of of words and it is not of finite type). It is sofic in the sense of Definition \ref{thesecond} because it is the image by the morphism $\pi:(x,y)\mapsto y$, of the Markov subshift of the sequences $(\xi_n)_{n\in\mathbb N}\in\{(a,0), (b,0),(b,1),(c,0)\}^{\mathbb N}$ such that $\xi_n\xi_{n+1}\in\{(a,0)(b,0),(a,0)(b,1),(b,0)(a,0),(b,1)(c,0),(c,0)(b,0),(c,0)(b,1)$ for any $n$. The graph of this Markov subshift is:

\includegraphics[trim = 0 180 0 228,clip,scale=0.4]{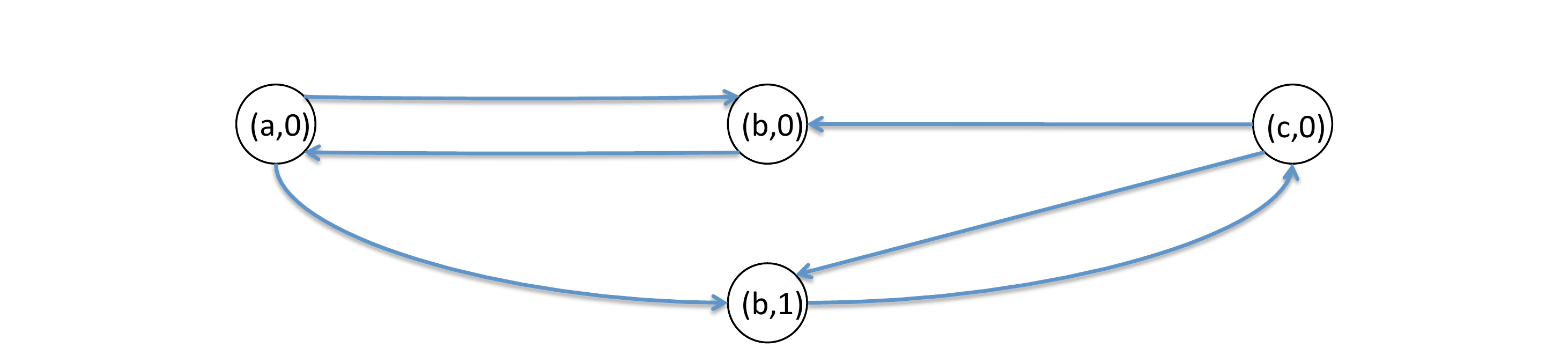}
\end{exmp}

\begin{exmp}\label{exmp2}
A example of sofic subshift in the sense of Definition \ref{thesecond}, is the image of the Markov subshift associated to the graph:

\includegraphics[trim = 0 200 0 200,clip,scale=0.4]{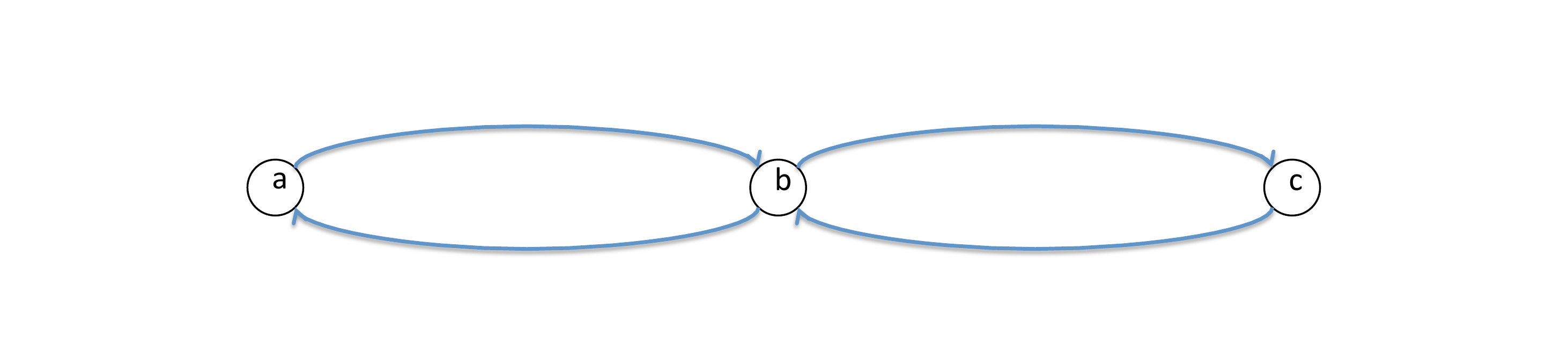}

by the letter-to-letter morphism $\varphi:\{a,b,c\}^{\mathbb N}\to\{0,1\}^{\mathbb N}$ associated to $\psi:\{a,b,c\}\to\{0,1\}$, $\psi(a)=0$, $\psi(b)=0$, $\psi(c)=1$ (as in Example \ref{exmp1}, it excludes the words $10^{2i}1$, $i\in\mathbb N\cup\{0\}$). It is also a sofic subshift in the sense of Definition \ref{la1}, because it is recognizable by the following automaton constructed from the graph (each arrow with initial state $x$ has label $\psi(x)$):

\includegraphics[trim = 0 200 0 200,clip,scale=0.4]{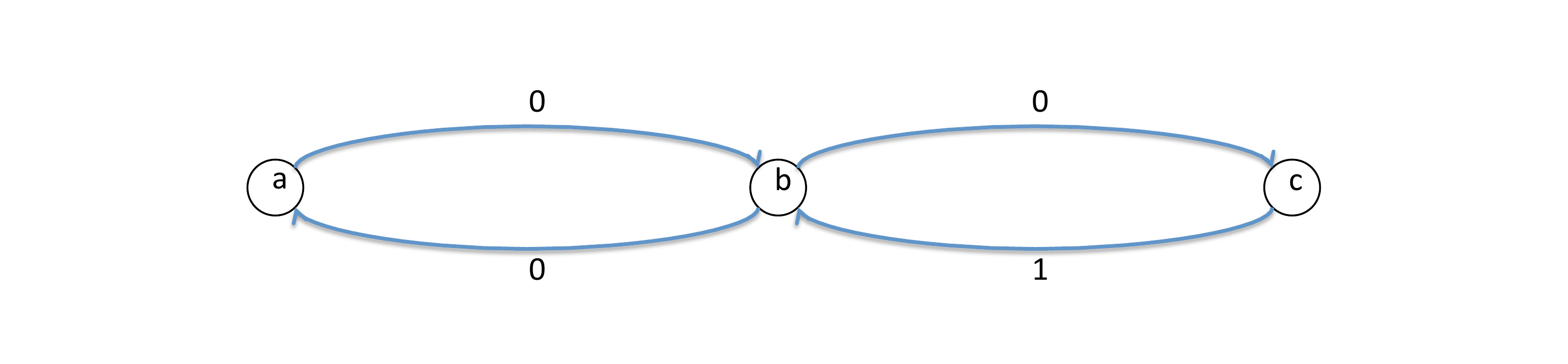}

Notice that, if we associate to this automaton a Markov subshift and a morphism by the same method as in Example \ref{exmp1}, we recover the initial Markov subshift and morphism of Example \ref{exmp2}.
\end{exmp}

\section{Markov, sofic and linearly representable measures}\label{Mslr}

\begin{defn}\label{measures}
(i) We call a (homogeneous) Markov probability measure (not necessarily shift-invariant), a measure $\mu$ on the product set $\{0,1,\dots,b-1\}^{\mathbb N}$ defined by setting, for any cylinder set $[\omega_1\dots\omega_n]=\{(\xi_i)_{i\in\mathbb N}\;:\;\xi_1\dots\xi_n=\omega_1\dots\omega_n\}$,
\begin{equation}\label{Mark}
\mu[\omega_1\dots\omega_n]=p_{\omega_1}p_{\omega_1\omega_2}
\dots p_{\omega_{n-1}\omega_n}
\end{equation}
where $p=\begin{pmatrix}p_0&\dots&p_{b-1}\end{pmatrix}$ is a positive probability vector and $P=\begin{pmatrix}p_{00}&\dots&p_{0(b-1)}\\\vdots&\ddots&\vdots\\p_{(b-1)0}&\dots&p_{(b-1)(b-1)}\end{pmatrix}$ a nonnegative stochastic matrix.

Clearly the support of $\mu$ is a Markov subshift, and $\mu$ is shift-invariant (or stationary) iff $p$ is a left eigenvector of $P$.

(ii) A probability measure on $\{0,1,\dots,b'-1\}^{\mathbb N}$ is called sofic if it is the image of a Markov probability measure by a continuous morphism $\varphi=\{0,1,\dots,b-1\}^{\mathbb N}\to\{0,1,\dots,b'-1\}^{\mathbb N}$. This morphism can be chosen letter-to-letter: $\varphi((\omega_n)_{n\in\mathbb N})=(\psi(\omega_n))_{n\in\mathbb N}$.

(iii) According to \cite{BP} we say that a probability measure $\eta$ on $\{0,1,\dots,b-1\}^{\mathbb N}$ is linearly representable if there exist a set of $r$-dimensional nonnegative row vectors $\{R_0,\dots,R_{b-1}\}$, a set of $r\times r$ nonnegative matrices $\mathcal M=\{M_0,\dots,M_{b-1}\}$ and a positive $r$-dimensional column vector $C$, satisfying both conditions
\begin{equation}\label{conditions}
\Big(\sum_iR_i\Big)C=1\quad\hbox{and}\quad\Big(\sum_iM_i\Big)C=C,
\end{equation}
and such that
\begin{equation}\label{M}
\eta[\omega_1\dots\omega_n]=R_{\omega_1}M_{\omega_2}
\dots M_{\omega_n}C.
\end{equation}
\end{defn}

\begin{rem}
If $\sum_iM_i$ is irreducible and if $R_i=RM_i$ for any $i$, where $R$ is the positive left eigenvector of $\sum_iM_i$ such that $RC=1$, the measure $\eta$ defined in (\ref{M}) is $\sigma$-invariant and
$$
\eta[\omega_1\dots\omega_n]=RM_{\omega_1}
\dots M_{\omega_n}C.
$$
\end{rem}

\begin{thm}\label{Msofic}
A probability measure on $\{0,1,\dots,b-1\}^{\mathbb N}$ is  sofic if and only if it is linearly representable.
\end{thm}

\begin{proof}The method we use is the one of \cite{BP}. First we note that the Markov measures are linearly representable because the formula (\ref{Mark}) is equivalent to
\begin{equation}\label{Markovmat}
\mu[\omega_1\dots\omega_n]=\pi_{\omega_1}P_{\omega_2}
\dots P_{\omega_n}\begin{pmatrix}1\\\vdots\\1\end{pmatrix}
\end{equation}
where
$$
\begin{array}{rclrclr}\pi_0&=&\begin{pmatrix}p_0&0&\dots&0\end{pmatrix},&\pi_1&=&\begin{pmatrix}0&p_1&\dots&0\end{pmatrix},&\dots\\
P_0&=&\begin{pmatrix}p_{00}&0&\dots&0\\\vdots&\vdots&\ddots&\vdots\\p_{(b-1)0}&0&\dots&0\end{pmatrix},&P_1&=&\begin{pmatrix}0&p_{01}&\dots&0\\\vdots&\vdots&\ddots&\vdots\\0&p_{(b-1)1}&\dots&0\end{pmatrix},&\dots
\end{array}
$$

Now any sofic measure $\nu$ defined from a Markov measure $\mu$ and a map
$$
\psi:\{0,1,\dots,b-1\}\to\{0,1,\dots,b'-1\}
$$
is linearly representable because (\ref{Markovmat}) implies
\begin{equation}\label{Sof}
\begin{array}{rcl}
\nu[\omega'_1\dots\omega'_n]&=&\sum_{\omega_1\in\psi^{-1}(\omega'_1)}\dots\sum_{\omega_n\in\psi^{-1}(\omega'_n)}\mu[\omega_1\dots\omega_n]\\
&=&(\sum_{\omega_1\in\psi^{-1}(\omega'_1)}\pi_{\omega_1})\dots(\sum_{\omega_n\in\psi^{-1}(\omega'_n)}P_{\omega_n})\begin{pmatrix}1\\\vdots\\1\end{pmatrix}\\&=&R_{\omega_1}M_{\omega_2}
\dots M_{\omega_n}C\end{array}
\end{equation}
where, for $0\le i'\le b'$, the row vector $R_{i'}:=\sum_{i\in\psi^{-1}(i')}\pi_i$, the matrix $M_{i'}:=\sum_{i\in\psi^{-1}(i')}P_i$ and the column vector $C:=\begin{pmatrix}1\\\vdots\\1\end{pmatrix}$  satisfy both conditions in (\ref{conditions}).

Conversely let $\eta$ be a linearly representable measure, so there exists some $r$-dimensional row vectors $R_0,\dots,R_{b-1}$, some $r\times r$ matrices $M_0,\dots,M_{b-1}$ and $r$-dimensional column vector $C$ satisfying (\ref{conditions}) and~(\ref{M}). Let $\Delta$ be the diagonal matrix whose diagonal entries are the entries of~$C$; setting $R'_i:=R_i\Delta$, $M'_i:=\Delta^{-1}M_i\Delta$ and $C':=\Delta^{-1}C$, the entries of $C'$ are $1$ and
$$
\eta[\omega_1\dots\omega_n]=R'_{\omega_1}M'_{\omega_2}
\dots M'_{\omega_n}C'.
$$
Setting
$$
\begin{array}{rclrclr}R''_0&:=&\begin{pmatrix}R'_0&0&\dots&0\end{pmatrix},&R''_1&:=&\begin{pmatrix}0&R'_1&\dots&0\end{pmatrix},&\dots\\
M''_0&:=&\begin{pmatrix}M'_0&0&\dots&0\\\vdots&\vdots&\ddots&\vdots\\M'_0&0&\dots&0\end{pmatrix},&M''_1&:=&\begin{pmatrix}0&M'_1&\dots&0\\\vdots&\vdots&\ddots&\vdots\\0&M'_1&\dots&0\end{pmatrix},&\dots\\C''&:=&\begin{pmatrix}C'\\\vdots\\C'\end{pmatrix}&&&
\end{array}
$$
we have again
$$
\eta[\omega_1\dots\omega_n]=R''_{\omega_1}M''_{\omega_2}
\dots M''_{\omega_n}C''.
$$
From (\ref{conditions}), $p:=\sum_iR''_i$ is a probability vector and $P:=\sum_iM''_i$ is a stochastic matrix, they define a Markov probability measure $\mu$ on the product set $\{0,1,\dots,rb-1\}^{\mathbb N}$, and $\eta$ is sofic because it is the image of $\mu$ by the morphism defined from the map $\psi:i\mapsto\left\lfloor\frac ir\right\rfloor$, $i\in\{0,1,\dots,rb-1\}$.\hfill\end{proof}

\section{Level sets and density spectrum associated to a map $f:\mathbb N\to\mathbb R^*_+$}

\subsection{Position of the problem}
This section is independent of the previous. In Theorem \ref{sets} below we don't make any hypothesis on the map $f:\mathbb N\to\mathbb R^*_+$ but, in the examples we consider later, $f$ has a polynomial rate of growth:
\begin{equation}\label{pol}
-\infty<\alpha_1:=\liminf_{n\to\infty}\frac{\log f(n)}{\log n}\le\alpha_2:=\limsup_{n\to\infty}\frac{\log f(n)}{\log n}<+\infty.
\end{equation}
The purpose is to associate to any $\alpha\in[\alpha_1,\alpha_2]$, a set of positive integers $\mathcal E(\alpha)$ as large as possible such that 
\begin{equation}\label{order}
\lim_{n\in\mathcal E(\alpha),\ n\to\infty}\frac{\log f(n)}{\log n}=\alpha
\end{equation}
(the notation $\displaystyle\lim_{n\in E,\ n\to\infty}u_n$ stands for $\displaystyle\lim_{k\to\infty}u_{n_k}$ where $E=\{n_1,n_2,\dots\}$ with $n_1<n_2<\dots$). Then we call "level sets" the sets $\mathcal E(\alpha)$, and "density spectrum" the map which associates to $\alpha$ the density of $\mathcal E(\alpha)$. Because of the following remark, we do not use the natural densities defined for any $S\subset\mathbb N$ by
$$
\begin{array}{l}\displaystyle\hbox{d}_-(S):=\liminf_{N\to\infty}\frac{\#S\cap[1,N)}N\\\displaystyle\hbox{d}_+(S):=\limsup_{N\to\infty}\frac{\#S\cap[1,N)}N\end{array}
$$
but the exponential densities defined by
$$
\begin{array}{l}\displaystyle\hbox{d}_-^{\rm{exp}}(S):=\liminf_{N\to\infty}\frac{\log\big(\#S\cap[1,N)\big)}{\log N}\\\displaystyle\hbox{d}_+^{\rm{exp}}(S):=\limsup_{N\to\infty}\frac{\log\big(\#S\cap[1,N)\big)}{\log N}.\end{array}
$$

\begin{rem}
One cannot expect to find for any $\alpha$ in a non trivial interval $[\alpha_1,\alpha_2]$, a set $\mathcal E(\alpha)$ with positive natural density $\hbox{d}_-(\mathcal E(\alpha))$ and satisfying the condition (\ref{order}). Indeed if we have $\hbox{d}_-(\mathcal E(\alpha))>0$ for any $\alpha$ in a non countable set $A$, there exists $\varepsilon>0$ such that $\hbox{d}_-(\mathcal E(\alpha))\ge\varepsilon$ for infinitely many values of $\alpha$. When $\alpha\ne\beta$, the set $\mathcal E(\alpha)\cap\mathcal E(\beta)$ is obviously finite, so by removing a finite number of elements from one of these sets, $\mathcal E(\alpha)$ and $\mathcal E(\beta)$ are disjoint. One can construct in this way a sequence of disjoint sets of positive integers $\mathcal E(\alpha_n)$ such that $\hbox{d}_-(\mathcal E(\alpha_n))\ge\varepsilon$, but this is in contradiction with the following inequalities (deduced from the general formula $\liminf_{n\to\infty}(u_n+v_n)\ge\liminf_{n\to\infty}(u_n)+\liminf_{n\to\infty}(v_n)$):
$$
1\ge \hbox{d}_-\big(\bigcup_{n\in\mathbb N}\mathcal E(\alpha_n)\big)\ge\sum_{n\in\mathbb N}\hbox{d}_-(\mathcal E(\alpha_n)).
$$
\end{rem}

\begin{thm}\label{sets}
We associate to any map $f:\mathbb N\to\mathbb R^*_+$ and to any $\alpha\ge0$, $\varepsilon>0$, the set
\begin{equation}\label{level}
\mathcal E(\alpha,\varepsilon):=\Big\{n\in\mathbb N\;:\;\alpha-\varepsilon\le\frac{\log f(n)}{\log n}\le\alpha+\varepsilon\Big\}.
\end{equation}
There exist some integers $1=N_1<N_2<\dots$ such that the set
\begin{equation}\label{Ealpha}
\mathcal E(\alpha):=\bigcup_{k\in\mathbb N}\big(\mathcal E(\alpha,1/k)\cap[N_k,N_{k+1},)\big),
\end{equation}
has densities
 \begin{equation}\label{fourdensities}
\begin{array}{l}
\displaystyle\hbox{d}_-(\mathcal E(\alpha))=\lim_{\varepsilon\to0}\hbox{d}_-(\mathcal E(\alpha,\varepsilon))=\inf_{\varepsilon>0}\hbox{d}_-(\mathcal E(\alpha,\varepsilon))\\\displaystyle\hbox{d}_-^{\rm{exp}}(\mathcal E(\alpha))=\lim_{\varepsilon\to0}\hbox{d}_-^{\rm{exp}}(\mathcal E(\alpha,\varepsilon))=\inf_{\varepsilon>0}\hbox{d}_-^{\rm{exp}}(\mathcal E(\alpha,\varepsilon))\\
\displaystyle\hbox{d}_+(\mathcal E(\alpha))=\lim_{\varepsilon\to0}\hbox{d}_+(\mathcal E(\alpha,\varepsilon))=\inf_{\varepsilon>0}\hbox{d}_+(\mathcal E(\alpha,\varepsilon))\\
\displaystyle\hbox{d}_+^{\rm{exp}}(\mathcal E(\alpha))=\lim_{\varepsilon\to0}\hbox{d}_+^{\rm{exp}}(\mathcal E(\alpha,\varepsilon))=\inf_{\varepsilon>0}\hbox{d}_+^{\rm{exp}}(\mathcal E(\alpha,\varepsilon)).
\end{array}
\end{equation}
If $\mathcal E(\alpha)$ is not finite, one has $\displaystyle\lim_{n\in\mathcal E(\alpha),\ n\to\infty}\frac{\log f(n)}{\log n}=\alpha$ and one says that $\mathcal E(\alpha)$ is a level set associated to $\alpha$. One also says that $\hbox{d}_-^{\rm{exp}}(\mathcal E(\cdot))$ and $\hbox{d}_+^{\rm{exp}}(\mathcal E(\cdot))$ are \underline{the} density spectrums of $f$, because they don't depend on the construction of $\mathcal E(\alpha)$ and depend only on the exponential densities of the sets $\mathcal E(\alpha,\varepsilon)$.
\end{thm}

The second subsection, independent on the first, will be used later to prove the theorem.

\subsection{A general lemma about the subsets of $\mathbb N$}

Given a monotonic sequence $(E_k)_k$ of subsets of $\mathbb N$, let us construct a subset $E$ whose densities are the limits of the ones of $E_k$.

\begin{lem}\label{glem}
Let $(E_k)_k$ be a sequence of subsets of $\mathbb N$, non-increasing or non-decreasing for the inclusion. There exist  some integers $1=N_1<N_2<\dots$ such that the set
\begin{equation}\label{Ecup}
E=\bigcup_{k\in\mathbb N}\big(E_k\cap[N_k,N_{k+1})\big)
\end{equation}
has densities
\begin{equation}\label{densi}
\begin{array}{l}
\displaystyle\hbox{d}_-(E)=\ell_-:=\lim_{k\to\infty}\hbox{d}_-(E_k),\quad\hbox{d}_-^{\rm{exp}}(E)=\ell_-^{\rm{exp}}:=\lim_{k\to\infty}\hbox{d}_-^{\rm{exp}}(E_k),\\\displaystyle\hbox{d}_+(E)=\ell_+:=\lim_{k\to\infty}\hbox{d}_+(E_k),\quad\hbox{d}_+^{\rm{exp}}(E)=\ell_+^{\rm{exp}}:=\lim_{k\to\infty}\hbox{d}_+^{\rm{exp}}(E_k).
\end{array}
\end{equation}
\end{lem}

\begin{proof}Let us define the integers $1=N_1<N_2<\dots$ by induction: we suppose that we know the value of $N_k$ for some $k$, and we search $N_{k+1}$ large enough in view to obtain (\ref{densi}). We have $E'_k\subset E_k\subset E''_k$ with
$$
E'_k:=E_k\cap[N_k,\infty)\hbox{ and }E''_k:=E_k\cup[1,N_k),
$$
and $E'_k,E''_k$ have same densities as $E_k$. So, by definition of "limitinf" and "limitsup", we can chose $N_{k+1}$ such that, for all $N\ge N_{k+1}$,
\begin{equation}\label{condi}
\begin{array}{l}
N\big(\hbox{d}_-(E_k)-\frac1k\big)\le\#E'_k\cap[1,N)\le\#E''_k\cap[1,N)\le N\big(\hbox{d}_+(E_k)+\frac1k\big),\\
N^{\hbox{$\scriptstyle\rm d$}_-^{\rm{exp}}(E_k)-\frac1k}\le\#E'_k\cap[1,N)\le\#E''_k\cap[1,N)\le N^{\hbox{$\scriptstyle\rm d$}_+^{\rm{exp}}(E_k)+\frac1k}\\
N\big(\hbox{d}_-(E_{k+1})-\frac1k\big)\le\#E_{k+1}\cap[1,N)\le N\big(\hbox{d}_+(E_{k+1})+\frac1k\big),\\
N^{\hbox{$\scriptstyle\rm d$}_-^{\rm{exp}}(E_{k+1})-\frac1k}\le\#E_{k+1}\cap[1,N)\le N^{\hbox{$\scriptstyle\rm d$}_+^{\rm{exp}}(E_{k+1})+\frac1k}.
\end{array}
\end{equation}
Using again the definition of "limitinf" and "limitsup", we can impose a supplementary conditions to $N_{k+1}$, according to the value of $k$ mod.~$4$:
\begin{equation}\label{condisuppl}
\begin{array}{ll}
\hbox{if }k\equiv0\hbox{ mod. }4,&\#E''_k\cap[1,N_{k+1})\le N_{k+1}\big(\hbox{d}_-(E_k)+\frac1k\big)\\
\hbox{if }k\equiv1\hbox{ mod. }4,&\#E''_k\cap[1,N_{k+1})\le N_{k+1}^{\hbox{$\scriptstyle\rm d$}_-^{\rm{exp}}(E_k)+\frac1k}\\
\hbox{if }k\equiv2\hbox{ mod. }4,&N_{k+1}\big(\hbox{d}_+(E_k)-\frac1k\big)\le\#E'_k\cap[1,N_{k+1})\\
\hbox{if }k\equiv3\hbox{ mod. }4,&N_{k+1}^{\hbox{$\scriptstyle\rm d$}_+^{\rm{exp}}(E_k)-\frac1k}\le\#E'_k\cap[1,N_{k+1}).\end{array}
\end{equation}
Since
$$
E'_k\cap[1,N_{k+1})\subset E\cap[1,N_{k+1})\subset E''_k\cap[1,N_{k+1})
$$
we deduce from (\ref{condi}) (applied to $N=N_{k+1}$) and (\ref{condisuppl}) that, according to the value of $k$ mod.~$4$,
\begin{equation}\label{finalimits}
\begin{array}{l}
\displaystyle\lim_{k\equiv0\atop k\to\infty}\frac{\#E\cap[1,N_{k+1})}{N_{k+1}}=\ell_-,\ \lim_{k\equiv1\atop k\to\infty}\frac{\log\big(\#E\cap[1,N_{k+1})\big)}{\log N_{k+1}}=\ell_-^{\rm{exp}},\\\displaystyle\lim_{k\equiv2\atop k\to\infty}\frac{\#E\cap[1,N_{k+1})}N_{k+1}=\ell_+,\ \lim_{k\equiv3\atop k\to\infty}\frac{\log\big(\#E\cap[1,N_{k+1})\big)}{\log N_{k+1}}=\ell_+^{\rm{exp}}.\end{array}
\end{equation}
For any $k\in\mathbb N$ and $N_{k+1}<N\le N_{k+2}$ we have the following inclusions, if the sequence of sets $(E_k)_k$ is non-increasing:
\begin{equation}\label{inclu1}
E_{k+1}\cap[1,N)\subset E\cap[1,N)\subset E''_k\cap[1,N)
\end{equation}
while, if the sequence of sets $(E_k)_k$ is non-decreasing:
\begin{equation}\label{inclu2}
E'_k\cap[1,N)\subset E\cap[1,N)\subset E_{k+1}\cap[1,N).
\end{equation}
Now(\ref{densi}) follows from (\ref{inclu1}), (\ref{inclu2}), (\ref{condi}) and (\ref{finalimits}).
\end{proof}

\begin{rem}
If one remove the hypothesis that $(E_k)_k$ is monotonic, it may happen that the lower (resp. upper) densities of the set $E$ defined by (\ref{Ecup}) are smaller (resp. larger) that the limitinf (resp. the limitsup) of the corresponding densities of $E_k$, even if the sequence $(N_k)_k$ increases quickly, so the method used for the proof of Lemma \ref{glem} do not apply. Suppose for instance that, for some positive integer $\kappa$, the set $E_1=\dots=E_{\kappa-1}$ has lower density $1/2$ and more precisely
$$
\#\big(E\cap[1,N_\kappa)\big)=\#\big(E_1\cap[1,N_\kappa)\big)=\lfloor N_\kappa/2\rfloor.
$$
Suppose that the set $E_\kappa$ is distinct from $E_1$ but has also lower density $1/2$, with for instance $E_\kappa\cap[1,2N_\kappa)=[1,N_\kappa)$. Then if $N_{\kappa+1}\ge2N_\kappa$, the integer $N=2N_\kappa$ satisfy
$$
\#\big(E\cap[1,N)\big)=\#\big(E_1\cap[1,N_\kappa)\big)=\lfloor N/4\rfloor.
$$
About the lower exponential density, suppose now that $E_1=\dots=E_{\kappa-1}$ have lower exponential density $1/2$, more precisely
$$
\#\big(E\cap[1,N_\kappa)\big)=\#\big(E_1\cap[1,N_\kappa)\big)=\lfloor{N_\kappa}^{1/2}\rfloor.
$$
Suppose that $E_\kappa$ has lower exponential density $1/2$, with for instance $E_\kappa\cap[1,{N_\kappa}^2)=[1,N_\kappa)$. Then if $N_{\kappa+1}\ge{N_\kappa}^2$, the integer $N={N_\kappa}^2$ satisfies
$$
\#\big(E\cap[1,N)\big)=\#\big(E_1\cap[1,N_\kappa)\big)=\lfloor N^{1/4}\rfloor.
$$
\end{rem}

\subsection{Proof of Theorem \ref{sets}}

\begin{proof}Lemma \ref{glem} applies to the non-increasing sequence
$$
E_k=\mathcal E(\alpha,1/k).
$$
\hfill\end{proof}

\begin{exmp}\label{triv}
Let $f(n)=n^{1+\sin n}$ then, given $\alpha\in[0,2]$ and $\varepsilon>0$, $E(\alpha,\varepsilon)$ is the set of the integers $n$ such that $1+\sin n\in[\alpha-\varepsilon,\alpha+\varepsilon]$. It has a positive usual density hence it has exponential density $1$, as well as $\mathcal E(\alpha)$.

\includegraphics[trim=0 0 150 0,scale=0.4]{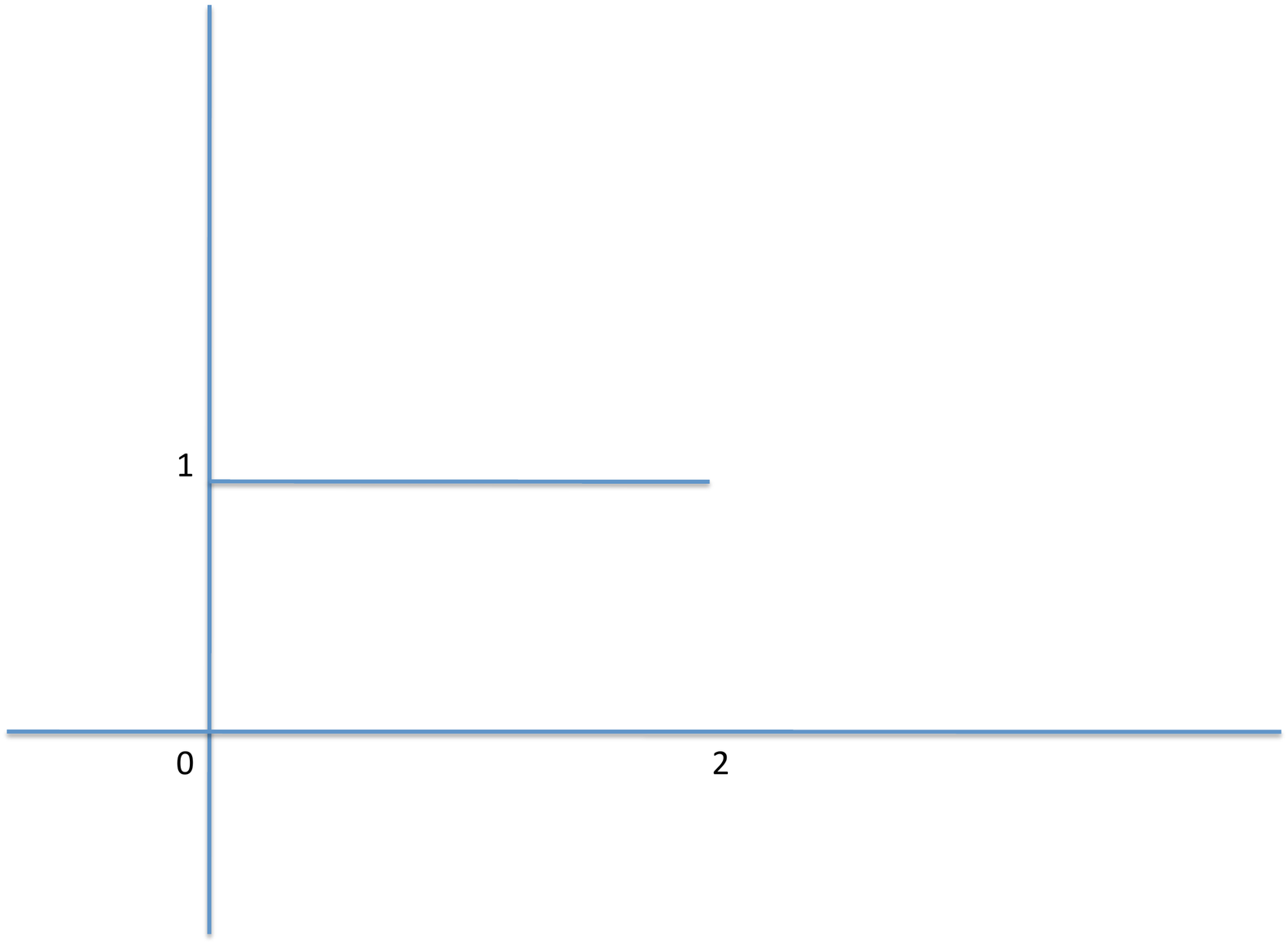}

\centerline{{\sc Figure 1.}\ \it Exponential density of $\mathcal E(\alpha)$ in function of $\alpha$ in Example \ref{triv}.}

\end{exmp}

\begin{exmp}\label{nontriv}
The number of representations of $n$ in base~$2$ with digit in $\{0,1,2\}$, defined by
\begin{equation}\label{base2}
f(n):\#\Big\{(\omega_i)_{i\ge0}\;:\;n=\sum_{i=0}^\infty\omega_i2^i,\ \omega_i\in\{0,1,2\}\Big\},
\end{equation}
has a polynomial rate of growth: (\ref{pol}) holds with $\alpha_1=0$ (because, for any~$k$, $f(2^k-1)=1$) and $\alpha_2=\log\frac{1+\sqrt5}2/\log2$ (see \cite[Corollary 6.10]{D}). By \cite[Theorem 19]{FLT} there exists a subset $S\subset\mathbb N$ of natural density $1$ -- and consequently exponential density $1$ -- such that the limit $\displaystyle\alpha_0:=\lim_{n\in S,\ n\to\infty}\frac{\log f(n)}{\log n}$ exists and is positive. Since $S\cap[N,\infty)\subset\mathcal E(\alpha_0,\varepsilon)$ for any $\varepsilon>0$ and for $N$ large enough, one has also
$$
{d}_\pm(\mathcal E(\alpha_0))={d}_\pm^{\rm{exp}}(\mathcal E(\alpha_0))=1.
$$
Nevertheless, let us deduce from \cite[Corollary 32]{FLT} that
\begin{equation}\label{alpha3}
\exists\alpha_3>\alpha_0,\ {d}_+^{\rm{exp}}(\mathcal E(\alpha_3))>0.
\end{equation}
From \cite[Corollary 32]{FLT}, for any $\alpha_0<\alpha'_0<\frac{\log3}{\log2}-1$ and $0<\beta'_0<\frac{\log3}{\log2}-1$, one has for $N$ large enough
\begin{equation}\label{ge}
\#\Big\{n<N\;;\;\frac{\log f(n)}{\log n}>\alpha'_0\Big\}\ge N^{\beta'_0}
\end{equation}
To prove (\ref{alpha3}) by contraposition, suppose that $\hbox{d}_+^{\rm{exp}}(\mathcal E(\alpha))<\beta'_0$ for any $\alpha\ge\alpha'_0$. Then by definition of $\mathcal E(\alpha)$, for each $\alpha\ge\alpha'_0$ there exists $\varepsilon_\alpha>0$ such that $\hbox{d}_+^{\rm{exp}}(\mathcal E(\alpha,\varepsilon_\alpha))<\beta'_0$. There exists $\beta_\alpha<\beta'_0$ such that, for $N$ large enough,
$$
\#\Big\{n<N\;;\;\alpha-\varepsilon_\alpha\le\frac{\log f(n)}{\log n}\le\alpha+\varepsilon_\alpha\Big\}\le N^{\beta_\alpha}.
$$
The open intervals $(\alpha-\varepsilon_\alpha,\alpha+\varepsilon_\alpha)$ cover the compact set $[\alpha'_0,3]$, so there exists a subcover of $[\alpha'_0,3]$ by $n_0$ intervals of this form. Let $\beta$ be the maximum of $\beta_\alpha$ for $\alpha$ being the center of such an interval. We have
\begin{equation}\label{le}
\#\Big\{n<N\;;\;\alpha'_0\le\frac{\log f(n)}{\log n}\le3\Big\}\le n_0N^{\beta}\quad(N\hbox{ large enough}).
\end{equation}
Since one check easily that $\frac{\log f(n)}{\log n}\le3$ for any $n\in\mathbb N$, (\ref{le}) is in contradiction with (\ref{ge}), hence (\ref{alpha3}) holds and more precisely there exists $\alpha_3\ge\alpha'_0$ and $\beta_3\ge\beta'_0$ such that $\hbox{d}_+^{\rm{exp}}(\mathcal E(\alpha_3))=\beta_3$.
\end{exmp}

\includegraphics[trim=0 65 150 0,clip,scale=0.45]{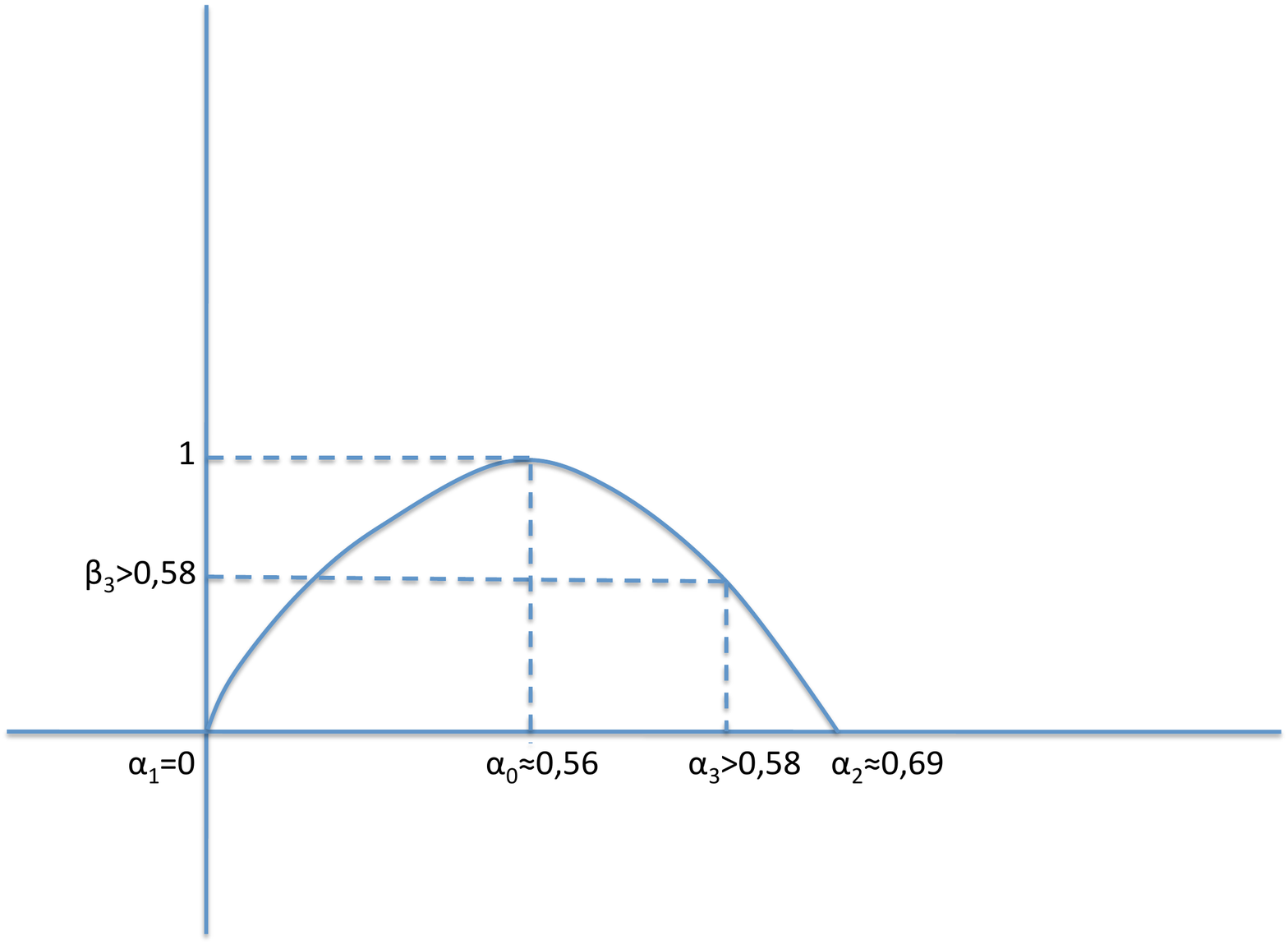}

\centerline{{\sc Figure 2.}\ \it Upper exponential density of $\mathcal E(\alpha)$ in function of $\alpha$ in Example \ref{nontriv}.}

\hskip50pt\includegraphics[trim=0 350 150 0,clip,scale=0.45]{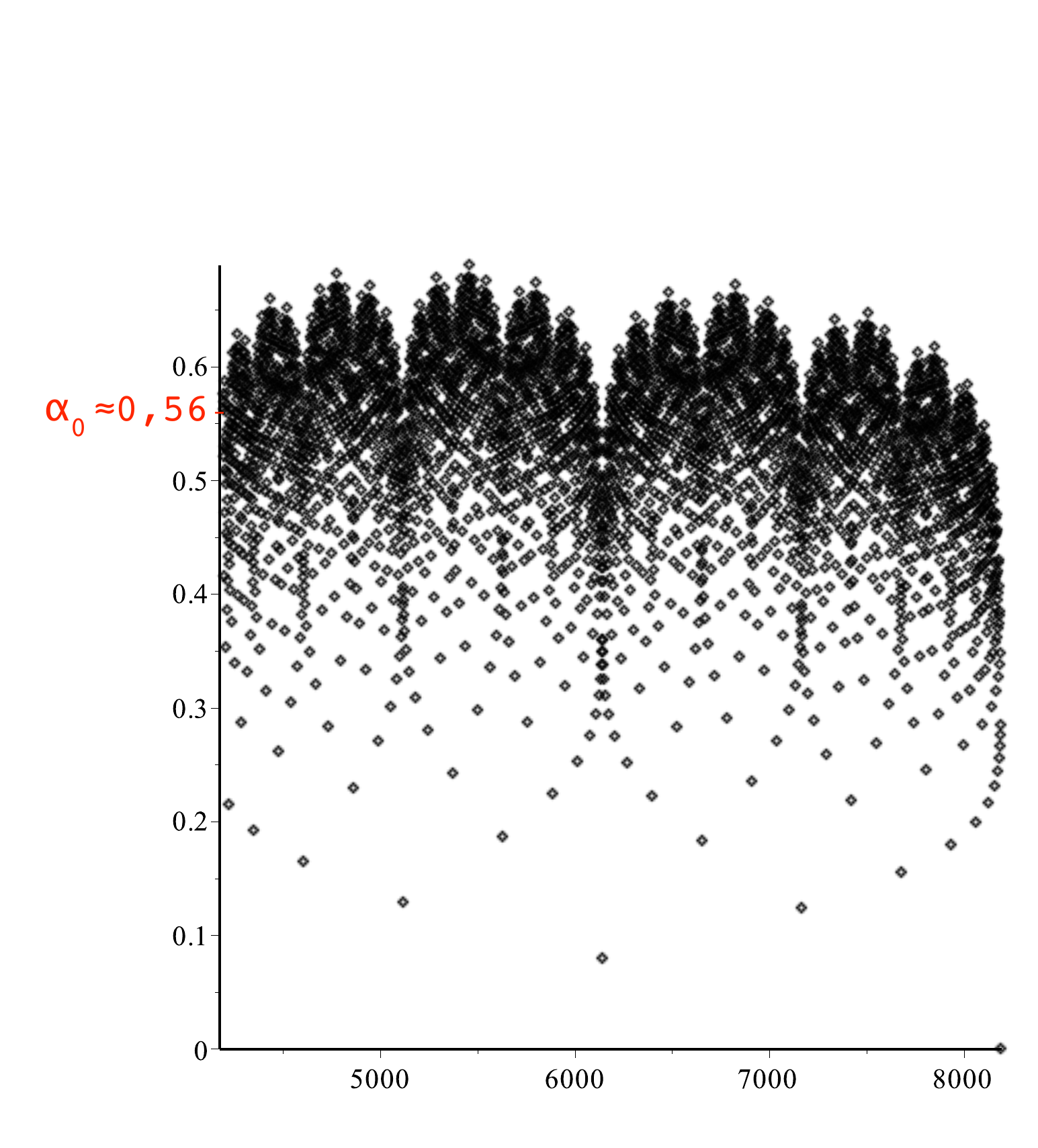}
 
\centerline{{\sc Figure 3.}\ \it The values of $\frac{\log f(n)}{\log n}$ for $n\in[2^{12},2^{13})$ in Example \ref{nontriv}.}


\section{Relation between singularity spectrum and density spectrum}\label{singandens}

Let $\eta$ be a probability measure on $\{0,1,\dots,b-1\}^{\mathbb N}$, the level sets $E(\alpha)$ are defined by
\begin{equation}\label{symbolicdef}
E(\alpha):=\Big\{(\omega_n)_{n\in\mathbb N}\;:\;\lim_{k\to\infty}\frac{\log\eta[\omega_1\dots\omega_k]}{\log(1/b^k)}=\alpha\Big\}
\end{equation}
and one calls the singularity spectrum, the map
$$
\alpha\mapsto\hbox{H-dim}(E(\alpha))
$$
where H-dim is the Hausdorff dimension, on the understanding that the distance between two sequences $(\omega_n)_{n\in\mathbb N}$ and $(\omega'_n)_{n\in\mathbb N}$ is $b^{1-\inf\{i\;:\;\omega_i\ne\omega'_i\}}$. Any ball is a cylinder set and the diameter of each cylinder set is
$$
\delta([\omega_1\dots\omega_k]):=b^{-k}.
$$

It is natural to associate to $\eta$ the function $f_\eta:\mathbb N\cup\{0\}\to[0,1]$ defined as follows from the expansion of $n$ in base $b$:
$$
f_\eta(n):=\eta[\omega_1\dots\omega_k]\quad\hbox{for any}\quad n=_{_b}\omega_1b^{k-1}+\dots+\omega_kb^0
$$
where the notation $=_{_b}$ means that $\forall i,\ \omega_i\in\{0,1,\dots,b-1\}$ and $\omega_1\ne0$. Notice that $f_\eta$ depends only on the restriction of $\eta$ to the set of the sequences $(\omega_n)_{n\in\mathbb N}$ such that $\omega_1\ne0$.

\begin{pro}\label{twolevel}
Let $\eta$ be a probability measure on $\{0,1,\dots,b-1\}^{\mathbb N}$ such that $\lim_{n\to\infty}\frac{\log\eta[0\omega_1\dots\omega_n]}{\log\eta[\omega_1\dots\omega_n]}=1$ for any $(\omega_n)_{n\in\mathbb N}$. The level sets $E(\cdot)$ (of the measure $\eta$) and $\mathcal E(\cdot)$ (of the function $f_\eta$) satisfy the inequality
$$
\hbox{\rm H-dim}(E(\alpha))\le\hbox{d}_-^{\rm{exp}}(\mathcal E(-\alpha)).
$$
\end{pro}

\begin{proof}
From the hypothesis on the probability $\eta$, a sequence $(\omega_n)_{n\in\mathbb N}$ with first term $\omega_1=0$ belongs to $E(\alpha)$ if and only if $(\omega_{n+1})_{n\in\mathbb N}$ do. The set $E'(\alpha)$ of the sequences $(\omega_n)_{n\in\mathbb N}\in E(\alpha)$ with first term $\omega_1\ne0$, has same Hausdorff dimension as $E(\alpha)$ because
$$
E(\alpha)=\bigcup_{n\ge0}\{0^n\}\times E'(\alpha)\quad\hbox{and}\quad\hbox{\rm H-dim}(\{0^n\}\times E'(\alpha))=\hbox{\rm H-dim}(E'(\alpha)).
$$

According to (\ref{fourdensities}) it is sufficient to prove for any $k\in\mathbb N$ the inequality $\hbox{H-dim}(E'(\alpha))\le\hbox{d}_-^{\rm{exp}}(\mathcal E(-\alpha,\frac1k))$. The integer $k$ is now fixed and, by definition of the limit~inf, there exists an infinite set $E\subset\mathbb N$ such that
\begin{equation}\label{extractedsequence}
\lim_{N\in E,\ N\to\infty}\frac{\log\big(\#\mathcal E(-\alpha,\frac1k)\cap[1,N)\big)}{\log N}=\hbox{d}_-^{\rm{exp}}\Big(\mathcal E\big(-\alpha,\frac1k\big)\Big).
\end{equation}
One can assume that the elements of $E$ have the form $N=b^i$ with $i\in\mathbb N$: indeed $N$ is in some interval $[b^{i(N)},b^{i(N)+1})$, the denominator $\log N$ in (\ref{extractedsequence}) is equivalent to $\log\big(b^{i(N)}\big)$, and the numerator is greater or equal to $\log\big(\#\mathcal E(-\alpha,\frac1k)\cap[1,b^{i(N)})\big)$.

Let $\omega\in E'(\alpha)$. There exists $\kappa\in\mathbb N$ such that
\begin{equation}\label{bydef}
\alpha-\frac1{2k}\le\frac{\log\eta[\omega_1\dots\omega_\kappa]}{\log(1/b^\kappa)}\le\alpha+\frac1{2k}
\end{equation}
and, since these inequalities are true for any $\kappa$ large enough, one can chose $\kappa=\kappa(\omega),\ $ such that $\kappa\ge2k\alpha+2$ and $b^\kappa\in E$. Let us prove that the integer
\begin{equation}\label{expansion?}
n=n(\omega):=\omega_1b^{\kappa-1}+\dots+\omega_\kappa b^0
\end{equation}
belongs to the level set $\mathcal E(-\alpha,\frac1k)$ if $k$ is large enough. Indeed the numerator in (\ref{bydef}) is $\log(f_\eta(n))$ because $\omega_1\ne0$, and
$$
-\frac{\log(f_\eta(n))}{\log n}=\frac{\log\eta[\omega_1\dots\omega_\kappa]}{\log(1/b^\kappa)}\ \frac{\log(b^\kappa)}{\log n}
$$
where $1\le\frac{\log(b^\kappa)}{\log n}\le\frac{\log(b^\kappa)}{\log(b^{\kappa-1})}=\frac\kappa{\kappa-1}$ and $(\alpha+\frac1{2k})\frac\kappa{\kappa-1}\le\alpha+\frac1k$ (consequence of the hypothesis $\kappa\ge2k\alpha+2$).

There exists a disjoint cover of $E'(\alpha)$ by a finite or countable family of cylinder sets $C_i$, each of the $C_i$ having the form $[\omega_1\dots\omega_{\kappa}]$ where $\kappa=\kappa(\omega)$ is defined in (\ref{bydef}). We consider the cylinder sets $C_i$ such that $\kappa(\omega)$ has a given value $\kappa_0$. The corresponding integers $n(\omega)$ are distinct and belong to $\mathcal E(-\alpha,\frac1k)\cap[1,b^{\kappa_0})$. By the hypotheses on $\kappa(\omega)$ one consider only the integers $\kappa_0\ge2k\alpha+2$ such that $b^{\kappa_0}\in E$; in particular, the larger is $k$, the larger is~$\kappa_0$. Let $\varepsilon>0$ and suppose that $k$ is large enough so that, applying (\ref{extractedsequence}) to $N=b^{\kappa_0}$ and setting $d_k=\hbox{d}_-^{\rm{exp}}\Big(\mathcal E\big(-\alpha,\frac1k\big)\Big))$,
$$
\frac{\log\big(\#\mathcal E(-\alpha,\frac1k)\cap[1,b^{\kappa_0})\big)}{\log b^{\kappa_0}}\le d_k+\varepsilon.
$$
Consequently, for any $s>d_k+\varepsilon$
$$
\sum_i\delta(C_i)^s\le\sum_{\kappa_0=0}^\infty b^{\kappa_0(d_k+\varepsilon)}b^{-\kappa_0s}=\frac1{1-b^{d_k+\varepsilon-s}},
$$
proving that $\hbox{H-dim}(E'(\alpha))\le d_k+\varepsilon$. Since it is true for any $k$ large enough and any $\varepsilon>0$, this implies $\hbox{H-dim}(E'(\alpha))\le\hbox{d}_-^{\rm{exp}}(\mathcal E(-\alpha))$.
\end{proof}

\section{Bernoulli convolution and number of representations in integral base}

\subsection{ Bernoulli convolution in integral base and related matrices}The Bernoulli convolution \cite{Erd39,PSS} in integral base $b\ge2$, associated to a positive probability vector $p=(p_0,\dots,p_{d-1})$ with $d\ge b$, is the probability measure $\eta=\eta_{b,p}$ defined by setting, for any interval~$I\subset\mathbb R$
\begin{equation}
\eta_{b,p}(I):=P_p\Big(\Big\{(\omega_k)_{k\in\mathbb N}\;:\;0\le\omega_k\le d-1,\ \sum_k\frac{\omega_k}{b^k}\in I\Big\}\Big)
\end{equation}
where $P_p$ the product probability defined on $\{0,\dots,d-1\}^{\mathbb N}$ from the probability vector $p$.

We define also, on the symbolic space $\{0,1,\dots,b-1\}^{\mathbb N}$, both probability measures
\begin{equation}
\begin{array}{l}\displaystyle\eta_{q,symb}[\varepsilon_1\dots\varepsilon_k]:=\frac{\eta(q+I_{\varepsilon_1\dots\varepsilon_k})}{\eta(q+[0,1))}\quad(q=0,1,2,\dots),\\
\displaystyle\eta_{sum,symb}[\varepsilon_1\dots\varepsilon_k]:=\sum_{q=0}^\infty\eta(q+I_{\varepsilon_1\dots \varepsilon_k}).\end{array}
\end{equation}
Let us prove that they are sofic.

\begin{rem}
The shift-invariant measure $\eta_{sum,symb}$ is the image of $P_p$ by the shift-commuting map $\varphi$ which associates to any $(\omega_k)_{k\in\mathbb N}$, the $b$-expansion of the fractional part of $\sum_k\frac{\omega_k}{b^k}$. Unfortunately $\varphi$ is discontinuous ($\lim_{n\to\infty}\varphi((b-1)^n\bar0)=\overline{b-1}\ne\varphi(\lim_{n\to\infty}(b-1)^n\bar0)=\bar0$), so this is not sufficient to prove that $\eta_{sum,symb}$ is sofic.
\end{rem}

Let the bi-infinite matrix
\begin{equation}\label{bi-i}
M_\infty:=\left(\begin{array}{ccccccccccc}\ddots&\vdots&\vdots&\reflectbox{$\ddots$}&\vdots&\ddots&\vdots&\ddots&\vdots&\vdots&\reflectbox{$\ddots$}\\\dots&0&p_{d-1}&\dots&p_{d-b-1}&\dots&p_0&\dots&0&0&\dots\\\dots&0&0&\dots&p_{d-1}&\dots&p_b&\dots&p_0&0&\dots\\\reflectbox{$\ddots$}&\vdots&\vdots&\ddots&\vdots&\reflectbox{$\ddots$}&\vdots&\reflectbox{$\ddots$}&\vdots&\vdots&\ddots\end{array}\right),
\end{equation}
where each row contains the same probability vector $\begin{pmatrix}p_{d-1}&\dots&p_0\end{pmatrix}$, shifted $b$ times to the right at the following row. We define in an unique way some matrices $M_0,\dots,M_{b-1}$, by setting that $M_0,\dots,M_{b-1}$ are submatrices of $M_\infty$ of size $a+1:=\left\lceil\frac{d-1}{b-1}\right\rceil$ and
\begin{equation}\label{thematrices}
\begin{array}{l}M_0:=\begin{pmatrix}p_0&0&\dots\\\vdots&\vdots&\ddots\end{pmatrix},\ M_1:=\begin{pmatrix}p_1&p_0&0&\dots\\\vdots&\vdots&\vdots&\ddots\end{pmatrix},\ \dots\ ,\\
M_{b-1}:=\begin{pmatrix}p_{b-1}&p_{b-2}&\dots\\\vdots&\vdots&\ddots\end{pmatrix}.\end{array}
\end{equation}
Assuming by convention that $p_i=0$ for $i\not\in\{0,\dots,d-1\}$, we can write
\begin{equation}\label{thematrix}
M_j=(p_{j+bq-q'})_{0\le q\le a\atop0\le q'\le a}=\begin{pmatrix}p_j&p_{j-1}&\dots&p_{j-a}\\p_{j+b}&p_{j+b-1}&\dots&p_{j+b-a}\\\vdots&\vdots&\ddots&\vdots\\p_{j+ab}&p_{j+ab-1}&\dots&p_{j+ab-a}&\end{pmatrix}.
\end{equation}

\begin{rem}\label{uniq}About the choice of the size of the matrices, $a=\left\lceil\frac{d-1}{b-1}\right\rceil-1$ is the largest integer such that the matrix $\sum_j(p_{j+bq-q'})_{0\le q\le a\atop0\le q'\le a}$ is irreducible, and the smallest integer such that its transpose is stochastic.\end{rem}

\begin{thm}\label{linrep}
The following formula gives the measure of the translated $b$-adic interval $q+I_{\varepsilon_1\dots \varepsilon_k}$ with $q\in\{0,1,\dots,a\}$, $\varepsilon_1,\dots,\varepsilon_k\in\{0,1,\dots,b-1\}$ and $I_{\varepsilon_1\dots\varepsilon_k}:=\left[\sum_{i=1}^k\frac{\varepsilon_i}{b^i},\ \sum_{i=1}^k\frac{\varepsilon_i}{b^i}+\frac1{b^k}\right)$:\begin{equation}\label{values}
\eta(q+I_{\varepsilon_1\dots\varepsilon_k})=E_{q}M_{\varepsilon_1}\dots M_{\varepsilon_k}C
\end{equation}
where $E_0,E_1,\dots,E_a$ are the canonical basis $(a+1)$-dimensional row vectors and $C$ the unique positive eigenvector of the irreducible matrix $\sum_i M_i$ such that $\sum_iE_iC=1$. Consequently the measures $\eta_{q,symb}$ and $\eta_{sum,symb}$ are linearly representable.
\end{thm}

\begin{proof}Let $I=I_{\varepsilon_1\dots\varepsilon_n}$ and $I'=I_{\varepsilon_2\dots\varepsilon_n}$. With the convention that $p_i=0$ for any $i\not\in\{0,1,\dots,d-1\}$, we have
$$
\begin{array}{rcl}\eta(q+I)&=&\displaystyle\sum_{i=0}^{d-1}P\Big(\big\{\omega_1=i\hbox{ and }\sum_k\frac{\omega_{k+1}}{b^k}\in q b+\varepsilon_1-i+I'\big\}\Big)\\&=&\displaystyle\sum_{i\in\mathbb Z}p_i\ \eta\left(q b+\varepsilon_1-i+I'\right)\\&=&\displaystyle\sum_{q'\in\mathbb Z}p_{q b+\varepsilon_1-q'}\ \eta\left(q'+I'\right).\end{array}
$$
In fact $q'$ belongs to $\{0,1,\dots,a\}$, otherwise $\eta\left(q'+I'\right)$ is null. Since the coefficients $p_{q b+\varepsilon_1-q'}$ for $q,q'\in\{0,1,\dots,a\}$, are the entries of $M_{\varepsilon_1}$,
$$
\begin{pmatrix}\eta(I)\\\eta(I+1)\\\vdots\\\eta(I+a)\end{pmatrix}=M_{\varepsilon_1}\begin{pmatrix}\eta(I')\\\eta(I'+1)\\\vdots\\\eta(I'+a)\end{pmatrix}
$$
and, by induction,
\begin{equation}\label{eta}
\begin{pmatrix}\eta(I)\\\eta(I+1)\\\vdots\\\eta(I+a)\end{pmatrix}=M_{\varepsilon_1}\dots M_{\varepsilon_k}C\quad\hbox{with }C:=\begin{pmatrix}\eta([0,1))\\\eta([1,2))\\\vdots\\\eta([a,a+1))\end{pmatrix}.
\end{equation}
In the particular case $k=1$ we have $I=I_{\varepsilon_1}$ and, making the sum in (\ref{eta}) for $\varepsilon_1=0,1,\dots,b-1$, we deduce that $C$ is a eigenvector of $\sum_i M_i$. Moreover $C$ is positive because the measure of any nontrivial subinterval of $\big[0,\frac{d-1}{b-1}\big]$ is positive, and $\sum_i E_iC=1$ because $[0,a+1]\supset\big[0,\frac{d-1}{b-1}\big]$ (the support of $\eta$), proving the unicity of $C$.
\hfill\end{proof}

\begin{cor}
The measures $\eta_{q,symb}$ and $\eta_{sum,symb}$ are sofic, as well as the measure $\eta'$ defined by
$$
\eta'[\varepsilon_1\dots\varepsilon_k]:=\eta_{sum,symb}[\varepsilon_k\dots\varepsilon_1].
$$
Moreover $\eta'$ is  the continuous image, by the map $i\mapsto\left\lfloor\frac i{a+1}\right\rfloor$, of the Markov measure on $\{0,1,\dots,(a+1)b-1\}^{\mathbb N}$ of transition matrix
$$
P_{\eta'}:=\begin{pmatrix}^t{M_0}&^t{M_1}&\dots&^t{M_{b-1}}\\\vdots&\vdots&\ddots&\vdots\\^t{M_0}&^t{M_1}&\dots&^t{M_{b-1}}\end{pmatrix}
$$
and initial probability vector $\begin{pmatrix}^tC&\dots&^tC\end{pmatrix}$.
\end{cor}

\begin{proof}
According to Theorems \ref{linrep} and \ref{Msofic}, $\eta_{q,symb}$ and $\eta_{sum,symb}$ are sofic. Formula (\ref{values}) gives a linear representation of $\eta'$:
$$
\eta'[\varepsilon_1\dots\varepsilon_k]=^t\hskip-2ptC\ ^t(M_{\varepsilon_1})\dots^t(M_{\varepsilon_k})\begin{pmatrix}1\\\vdots\\1\end{pmatrix}.
$$
According to the proof of the converse part of Theorem \ref{Msofic}, $P_{\eta'}$ is the transition matrix and $\begin{pmatrix}^tC&\dots&^tC\end{pmatrix}$ is the initial probability vector of the Markov chain associated to $\eta'$.\hfill\end{proof}

\subsection{The matrices $M_j$ are also related to the transducer of normalization \cite{FS} and to the number of representations of the integers in base $b$}

\begin{defn}(i) The normalization in base $b$, with $d\ge b$ digits, is the map
$$
\begin{array}{l}\begin{array}{rl}\frak n:&\{0,\dots,d-1\}^{\mathbb N}\to\{0,\dots,b-1\}^{\mathbb Z}\\&(\omega_i)_{i\in\mathbb N}\mapsto(\varepsilon_i)_{i\in\mathbb Z}\quad\hbox{such that}\end{array}\\\displaystyle\sum_{i=1}^\infty\frac{\omega_i}{b^i}=\sum_{i=-\infty}^\infty\frac{\varepsilon_i}{b^i}\Big(=\sum_{i=-(h-1)}^\infty\frac{\varepsilon_i}{b^i}\hbox{ with }\varepsilon_{-(h-1)}\ne0\hbox{ if }(\omega_i)_{i\in\mathbb N}\ne(0)_{i\in\mathbb N}\Big)\end{array}
$$
with the additional condition that the $\varepsilon_i$ are not eventually $b-1$.

(ii) The normalization in base $b$ also associates to any finite sequence $(\omega_{h-1},\dots,\omega_0)$ with terms in $\{0,\dots,d-1\}$ and distinct from $0^h$, the sequence $(\varepsilon_{k-1},\dots,\varepsilon_0)$ with terms in $\{0,\dots,b-1\}$ such that
\begin{equation}\label{normalized}
\omega_{h-1}b^{h-1}+\dots+\omega_0b^0=\varepsilon_{k-1}b^{k-1}+\dots+\varepsilon_0b^0\hbox{ and }\varepsilon_{k-1}\ne0.
\end{equation}
Notice that if we put $\omega_h=\omega_{h+1}=\dots=\omega_{k-1}=0$ we have
$$
\frak n(\omega_{k-1},\dots,\omega_0,0,0,\dots)=(\dots,0,0,\varepsilon_{k-1},\dots,\varepsilon_0,0,0,\dots).
$$

(iii) The number of $b$-representations of $n$ with digits in $\{0,\dots,d-1\}$~is
$$
\mathcal N(n):=\#\Big\{(\omega_i)_{i\ge0}\;:\;n=\sum_{i=0}^\infty\omega_ib^i,\ \omega_i\in\{0,\dots,d-1\}\Big\}.
$$
The number of $b$-representations of length $k$ is
$$
\mathcal N_k(n):=\#\Big\{(\omega_0,\dots,\omega_{k-1})\;:\;n=\sum_{i=0}^{k-1}\omega_ib^i,\ \omega_i\in\{0,\dots,d-1\}\Big\}.
$$
\end{defn}

Now we define the transducer $\mathcal T$ as follows:

\centerline{\includegraphics[trim=0 400 0 100,clip,scale=0.5]{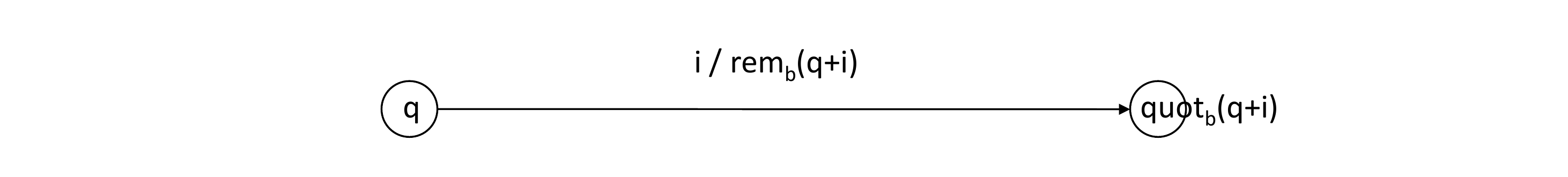}}

where $q$ belongs to the set of states $\{0,\dots,a\}$ with $a=\left\lceil\frac{d-1}{b-1}\right\rceil-1$, and quot$_b(q+i)$ and rem$_b(q+i)$ are respectively the quotient and the remainder of the Euclidean division of $q+i$ by~$b$.
It is the transpose of the transducer of normalization of Frougny and Sakarovitch in \cite[Propositions 2.2.5, 2.2.6 and Theorem 2.2.7]{FS}, which chose $\{-a,\dots,a\}$ as set of states.

\begin{pro}\label{rightleft}Let us consider the matrices $M_0,\dots,M_{b-1}$ defined by (\ref{thematrices}) or (\ref{thematrix}), in the case $p_0=\dots=p_{d-1}=\frac1d$.

(i) The transpose of $dM_j$ is the incidence matrix of the graph, whose set of vertices is $\{0,\dots,a\}$ and whose edges are some edges of $\mathcal T$: only the edges with output label~$j$.

(ii) Let $w=\varepsilon_1\dots\varepsilon_k\in\{0,\dots,b-1\}^k$ and $M_w=M_{\varepsilon_1}\dots M_{\varepsilon_k}$, we have
$$
d^kM_w=\begin{pmatrix}\mathcal N_k(n)&\mathcal N_k(n-1)&\dots&\mathcal N_k(n-a)\\\mathcal N_k(n+b^k)&\mathcal N_k(n+b^k-1)&\dots&\mathcal N_k(n+b^k-a)\\\vdots&\vdots&\ddots&\vdots\\\mathcal N_k(n+ab^k)&\mathcal N_k(n+ab^k-1)&\dots&\mathcal N_k(n+ab^k-a)\end{pmatrix}
$$
where $n=\varepsilon_1b^{k-1}+\dots+\varepsilon_kb^0$. In the first row, the $\mathcal N_k(n-i)$ are equal to $\mathcal N(n-i)$, so we have
\begin{equation}\label{etabpdeI}
\eta_{b,p}(I_{\varepsilon_1\dots\varepsilon_k})=\frac1{d^k}\begin{pmatrix}\mathcal N(n)&\mathcal N(n-1)&\dots&\mathcal N(n-a)\end{pmatrix}C.
\end{equation}

(iii) To normalize a finite sequence $(\xi_{h-1},\dots,\xi_0)\in\{0,\dots,d-1\}^h$, one enters in $\mathcal T$ the digits $\xi_0,\dots,\xi_{h-1},0,0,\dots$ from the initial state~$0$, and one obtain in output the digits $\zeta_0,\dots,\zeta_{k-1}$ such that
\begin{equation}\label{expectedformula}
\xi_{h-1}b^{h-1}+\dots+\xi_0b^0=\zeta_{k-1}b^{k-1}+\dots+\zeta_0b^0.
\end{equation}
\end{pro}

\begin{proof}(i) In (\ref{thematrix}) we have $p_{j+bq-q'}=\frac1d$ when $j+bq-q'\in\{0,\dots,d-1\}$ that is, when a edge of output label $j$ relates $q'$ to $q$ in $\mathcal T$.

(ii) Consequently the entry of the $(q+1)^{\rm th}$ row and $(q'+1)^{\rm th}$ column of $^t(dM_{\varepsilon_k})\dots^t(dM_{\varepsilon_1})$ is the number of paths from $q$ to $q'$ whose output label is $\varepsilon_k\dots\varepsilon_1$. The input label $\omega_k\dots\omega_1$ and the states $q_k,\dots,q_0$ of such a path satisfy

\begin{equation}\label{DivEucl}
\begin{array}{l}q_k+\omega_k=bq_{k-1}+\varepsilon_k\ (\hbox{with }q_k=q)\\
q_{k-1}+\omega_{k-1}=bq_{k-2}+\varepsilon_{k-1}\\
\vdots\\
q_1+\omega_1=bq_0+\varepsilon_1\ (\hbox{with }q_0=q').
\end{array}
\end{equation}

We deduce $\sum_{i=1}^k(q_i+\omega_i)b^{k-i}=b\sum_{i=1}^kq_{i-1}b^{k-i}+\sum_{i=1}^k\varepsilon_ib^{k-i}$ and, after simplification,
\begin{equation}\label{reprk}
\sum_{i=1}^k\omega_ib^{k-i}=q'b^k-q+\sum_{i=1}^k\varepsilon_ib^{k-i}
\end{equation}
meaning that $\omega_1\dots\omega_k$ is a representation of length $k$ of $n+b^kq'-q$.

Conversely if (\ref{reprk}) holds, then (\ref{DivEucl}) holds with
$$
q_i=q'b^i+\sum_{j\le i}(\varepsilon_j-\omega_j)b^{i-j}.
$$
It remains to prove that (\ref{DivEucl}) implies $q_i\le a$ by descending induction. Since $q_k=q$ one has $q_k\le a$. If $q_i\le a$, the Euclidean divisions in (\ref{DivEucl}) imply
$$
q_{i-1}=\left\lfloor\frac{q_i+\omega_i}b\right\rfloor\le\left\lfloor\frac{a+d-1}b\right\rfloor
$$
where the r.h.s. is at most $a$ because $\left\lfloor\frac{a+d-1}b\right\rfloor\le\frac{a+d-1}b<\frac{\frac{d-1}{b-1}+d-1}b=\frac{d-1}{b-1}\le\left\lceil\frac{d-1}{b-1}\right\rceil=a+1$.

(iii) In particular, if (\ref{reprk}) holds with $q=q'=0$ then $\omega_k\dots\omega_1$ is the input label and $\varepsilon_k\dots\varepsilon_1$ the output label of a path from the state $0$ to the state $0$. Suppose that (\ref{expectedformula}) holds. Equivalently, (\ref{reprk}) holds with $q=q'=0$, with $\omega_1\dots\omega_k=0^{h-k}\xi_{h-1}\dots\xi_0$ and $\varepsilon_1\dots\varepsilon_k=\zeta_{k-1},\dots,\zeta_0$. This proves that, if we enter the digits $\xi_0,\dots,\xi_{h-1},0,0,\dots$ from the initial state $0$, the output digits are $\zeta_0,\dots,\zeta_{k-1}$.
\hfill\end{proof}

\begin{thm}\label{levels}
In the case $p_0=\dots=p_{d-1}=\frac1d$ the level sets of the measure $\eta_{b,p}$, that is, the sets
$$
E(\alpha):=\Big\{x\in\mathbb R\;:\;\lim_{r\to0}\frac{\log\eta_{b,p}\big((x-r,x+r)\big)}{\log r}=\alpha\Big\},
$$
are related to the level sets $\mathcal E(\cdot)$ of the function $\mathcal N$ by the inequality
\begin{equation}\label{analogic}
\hbox{\rm H-dim}(E(\alpha))\le\hbox{d}_-^{\rm{exp}}\Big(\mathcal E\Big(\frac{\log d}{\log b}-\alpha\Big)\Big).
\end{equation}
\end{thm}

\begin{proof}
We first consider the measure $\eta_{0,symb}$ on the symbolic space $\{0,\dots,b-1\}^{\mathbb N}$ and we prove that the level sets $E_{0,symb}(\alpha)$ of this measure, defined as in (\ref{symbolicdef}), satisfy (\ref{analogic}). The condition of Proposition \ref{twolevel} is satisfied: indeed (\ref{values}) implies $\frac{\eta_{0,symb}[0\varepsilon_1\dots\varepsilon_k]}{\eta_{0,symb}[\varepsilon_1\dots\varepsilon_k]}=\frac1d$ and thus $\lim_{k\to\infty}\frac{\log\eta_{0,symb}[0\varepsilon_1\dots\varepsilon_k]}{\log\eta_{0,symb}[\varepsilon_1\dots\varepsilon_k]}=1$. For any $n=_{_b}\varepsilon_1b^{k-1}+\dots+\varepsilon_kb^0$ one has $b^{k-1}\le n<b^k$ and consequently $k=k(n)=1+\left\lfloor\frac{\log n}{\log b}\right\rfloor$. By (\ref{etabpdeI}),
$$
\eta_{0,symb}[\varepsilon_1\dots\varepsilon_{k(n)}]=\frac1{d^{k(n)}\eta([0,1))}\begin{pmatrix}\mathcal N(n)&\dots&\mathcal N(n-a)\end{pmatrix}C.
$$
Let $f(n)=\eta_{0,symb}[\varepsilon_1\dots\varepsilon_{k(n)}]$, by Proposition \ref{twolevel} one has
$$
\hbox{\rm H-dim}\big(E_{0,symb}(\alpha)\big)\le\hbox{d}_-^{\rm{exp}}(\mathcal E_f(-\alpha)).
$$
The function $g(n):=\frac1{d^{k(n)}}\mathcal N(n)$ has the same density spectums as $f$ because, from \cite[Lemma~9~(ii)]{FLT}, there exists a positive constant $K$ such that
\begin{equation}\label{lem9}
\frac1{K\log n}\le\frac{\mathcal N(n)}{\mathcal N(n-1)}\le K\log n.
\end{equation}
Since $\lim_{n\to\infty}\big(\frac{\log\mathcal N(n)}{\log n}-\frac{\log g(n)}{\log n}\big)=\frac{\log d}{\log b}$, the level sets $\mathcal E(\cdot)$ of $\mathcal N$ satisfy $\hbox{d}_-^{\rm{exp}}\Big(\mathcal E\Big(\frac{\log d}{\log b}-\alpha\Big)\Big)=\hbox{d}_-^{\rm{exp}}(\mathcal E_g(-\alpha))=\hbox{d}_-^{\rm{exp}}(\mathcal E_f(-\alpha))$ and consequently
\begin{equation}\label{firstep}
\hbox{\rm H-dim}\big(E_{0,symb}(\alpha)\big)\le\hbox{d}_-^{\rm{exp}}\Big(\mathcal E\Big(\frac{\log d}{\log b}-\alpha\Big)\Big).
\end{equation}

Let us now prove by a classical method that the real $x=_{_b}\sum_{k=1}^\infty\frac{\varepsilon_k}{b^k}$ belongs to $E(\alpha)\cap(0,1)$ if and only if $(\varepsilon_k)_k$ belongs to $E_{0,symb}(\alpha)\setminus\{\bar0\}$. This is due to the fact that $I_{\varepsilon_1\dots\varepsilon_k}\subset\big(x-\frac1{b^k},x+\frac1{b^k}\big)$, and conversely $(x-r,x+r)\subset I_w\cup I_{w'}\cup I_{w''}$ holds for $\frac1{b^{k+1}}\le r<\frac1{b^k}$, the words $w'$ and $w''$ being respectively the lexicographical predecessor and the lexicographical successor of $w=\varepsilon_1\dots\varepsilon_k$. We use of course the relation (\ref{etabpdeI}) and the inequality (\ref{lem9}). Let us now prove that
\begin{equation}\label{EE}
\hbox{\rm H-dim}\big(E(\alpha)\cap[0,1)\big)=\hbox{\rm H-dim}\big(E_{0,symb}(\alpha)\big).
\end{equation}
To any cover of $E_{0,symb}(\alpha)$ corresponds a cover of $E(\alpha)\cap[0,1)$, and conversely to any cover of $E(\alpha)\cap[0,1)$ and to any interval $(a,a')$ of this cover, correspond two cylinder sets $[w],[w']$ such that $(a,a')\subset I_w\cup I_{w'}$, where the words $w,w'$ have length $k$ such that $\frac1{b^{k+1}}\le a'-a<\frac1{b^k}$. One deduce easily that (\ref{EE}) holds.

The measure $\eta_{b,p}$ is obviously symmetrical with respect to the middle of its support $\big[0,\frac{d-1}{b-1}\big]$. So $E(\alpha)$ is a symmetric subset of $\big[0,\frac{d-1}{b-1}\big]$ of $\eta_{b,p}$, and it remains to prove that $E(\alpha)\cap(q,q+1)=q+E(\alpha)\cap(0,1)$ for any positive integer $q<\frac12\frac{d-1}{b-1}$ (hence $q\ne a$). Given a positive integer $k_0$, we consider the $b$-adic intervals $I_{\varepsilon_1\dots \varepsilon_k}$ such that $k\ge k_0+a$ and $\varepsilon_1\dots \varepsilon_{k_0}\ne0^{k_0}$. The entries of $M_{\varepsilon_1}\dots M_{\varepsilon_{k_0+a}}$ are positive, except possibly the ones of its last row: this is due to the fact that the entries of $dM_j$ are at least equal to the ones of $\begin{pmatrix}1&1&0&0&\dots&0&0&0&0\\1&1&1&0&\dots&0&0&0&0\\\vdots&\vdots&\vdots&\vdots&\ddots&\vdots&\vdots&\vdots&\vdots\\0&0&0&0&\dots&0&1&1&1\\0&0&0&0&\dots&0&0&0&0\end{pmatrix}$, except the second entry of the first row of $dM_0$, this entry being $0$. Denoting by $M(k_0)$ the largest entry of the matrix (with positive integral entries) $d^{k_0+a}E_qM_{\varepsilon_1}\dots M_{\varepsilon_{k_0+a}}$ for any $q\in\{0,\dots,a-1\}$ and $\varepsilon_1\dots\varepsilon_{k_0+a}\in\{0,\dots,b-1\}^{k_0+a}$ such that $\varepsilon_1\dots \varepsilon_{k_0}\ne0^{k_0}$, the formula (\ref{values}) implies
$$
\frac1{M(k_0)}\le\frac{\eta(q+I_{\varepsilon_1\dots \varepsilon_k})}{\eta(I_{\varepsilon_1\dots \varepsilon_k})}\le M(k_0).
$$
Since it is true for any $k_0\in\mathbb N$ and $q\in\{0,\dots,a-1\}$, one deduce $E(\alpha)\cap(q,q+1)=q+E(\alpha)\cap(0,1)$ for $0<q<a$, and (\ref{analogic}) follows from (\ref{firstep}), (\ref{EE}) and from the symmetry of $E(\alpha)$.
\end{proof}

\section{The general framework in Pisot base}

The normalization map \cite{FS} in the integral or Pisot base $\beta>1$, associates to each sequence $(\omega_i)_{i\in\mathbb N}\in\{0,1,\dots,d-1\}^{\mathbb N}$ the sequence $(\varepsilon_i)_{i\in\mathbb Z}\in\{0,1,\dots,\lceil\beta\rceil-1\}^{\mathbb Z}$ satisfying both conditions:
\begin{equation}\label{omegaepsilon}
\begin{array}{l}\displaystyle\sum_{i\in\mathbb N}\frac{\omega_i}{\beta^i}=\sum_{i\in\mathbb Z}\frac{\varepsilon_i}{\beta^i}\\\displaystyle\forall i\in\mathbb Z,\ \sum_{j>i}\frac{\varepsilon_j}{\beta^j}<\frac1{\beta^i}\hbox{ (Parry $\beta$-admissibility condition \cite{P})}.\end{array}
\end{equation}

\subsection{The tranducer $\mathcal T$ associated to $\beta$ and $d$}\label{thetr}

The states of the transducer are the carries of the normalization of the sequences $(\omega_i)_{i\in\mathbb N}$ with digits in $\{0,1,\dots,d-1\}$. More precisely suppose that (\ref{omegaepsilon}) holds and put $\omega_i=0$ for any $i\le0$; then for each $i\in\mathbb Z$ the sum $\displaystyle\sum_{j>i}\frac{\omega_j}{\beta^j}$ is equal to $\displaystyle\sum_{j>i}\frac{\varepsilon_j}{\beta^j}$ plus a real number that we denote by $\displaystyle\frac{q_i}{\beta^i}$. So we call "the $i^{\rm th}$ carry", the real $q_i$ defined by both relations
\begin{equation}\label{carry}
\frac{q_i}{\beta^i}=\sum_{j>i}\frac{\omega_j-\varepsilon_j}{\beta^j}=\sum_{j\le i}\frac{\varepsilon_j-\omega_j}{\beta^j}.
\end{equation}
The first relation implies $q_i\in\ (-1,\alpha]$ with $\alpha:=\frac{d-1}{\beta-1}$. The second relation implies
\begin{equation}\label{Eucl-beta}
q_i=\beta q_{i-1}-\omega_i+\varepsilon_i
\end{equation}
and implies that $q_i$ belongs to the set
$$
S_{\beta,d}:=\left\{\sum_{j=0}^{i-1}\alpha_j\beta^j\;:\;i\in\mathbb N,\ \alpha_j\in(-d,\lceil\beta\rceil)\cap\mathbb Z,\right\}\cap\ (-1,\alpha].
$$
Garsia's separation lemma \cite{G} ensuring that $S_{\beta,d}$ is finite, we can chose $S_{\beta,d}$ as set of states of $\mathcal T$ and assume that the arrows have the form 
$q{{\omega\ /\ \varepsilon\atop\xrightarrow{\hspace*{50pt}}}\atop\phantom{}}\beta q-\omega+\varepsilon$. Notice that the arrows are in the opposite direction of the ones considered in integral base because, as can be seen for instance in Example \ref{ex}, the Euclidean division that we use for the normalization in integral base do not have analogue in non-integral base; more precisely, the relation $q'+\omega=\beta q+\varepsilon$ may hold for several values of $q$ when $q'$ and $\omega$ are fixed.

In practice we construct at the same time a suitable set of states $S'_{\beta,d}\subset S_{\beta,d}$ and the transducer: $S'_{\beta,d}$ is by definition the smallest set containing $0$ and containing $\beta q-\omega+\varepsilon$, whenever the three reals $q\in S'_{\beta,d}$, $\omega\in\{0,1,\dots,d-1\}$, $\varepsilon\in\{0,1,\dots,\lceil\beta\rceil-1\}$ satisfy the condition $\beta q-\omega+\varepsilon\in(-1,\alpha]$.

\begin{exmp}\label{ex}The transducer $\mathcal T$ when $\beta^2=3\beta-1$ and $d=3$:

\includegraphics[trim=0 230 0 40,clip,scale=0.42]{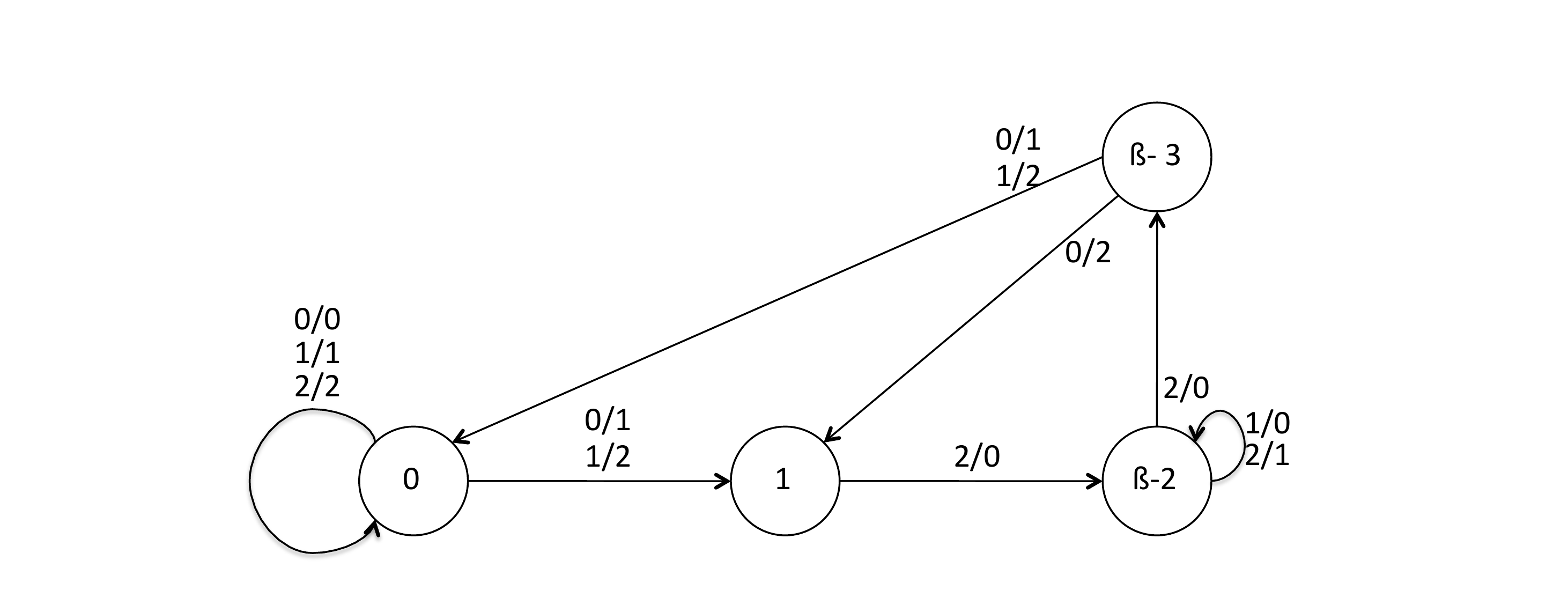}

For instance $3-\beta=-\beta^2+2\beta+2=\frac1\beta$ belongs to $S_{\beta,d}$ but do not belong to $S'_{\beta,d}$.

\end{exmp}

\subsection{How to normalize a sequence $(\omega_i)_{i\in\mathbb N}$?}

We first notice that in a non-integral base there do not exist a literal transducer of normalization. For instance in base $\beta=\frac{1+\sqrt5}2$, if we normalize from the left to the right a sequence of the form $(01)^nx$ without knowing if the digit $x$ is $0$ or~$1$, we obtain $(01)^n0$ if $x=0$ and $1(00)^n$ if $x=1$. So all the terms of the normalized sequence depend on the value of $x$, while the successive carries cannot depend on $x$ because $x$ is at the right. As well if we normalize from the right to the left a sequence of the form $0x(11)^n$ without knowing if the digit $x$ is $0$ or $1$, we obtain $0(10)^n0$ if $x=0$ and $(10)^n01$ if $x=1$.


\begin{pro}
Let $(\omega_i)_{i\in\mathbb N}\in\{0,1,\dots,d-1\}^{\mathbb N}$. The normalized sequence defined in (\ref{omegaepsilon}) is the unique $\beta$-admissible sequence which is the output label of a bi-infinite path of input label $\bar0\omega_1\omega_2\omega_3\dots$.
\end{pro}

\begin{proof}
If (\ref{omegaepsilon}) holds, as noted above the states $q_i$ defined by (\ref{carry}) satisfy (\ref{Eucl-beta}), so $(\varepsilon_i)_{i\in\mathbb Z}$ is the output label of a bi-infinite path of input label $\bar0\omega_1\omega_2\omega_3\dots$.

Conversely if $(\varepsilon_i)_{i\in\mathbb Z}$ is $\beta$-admissible and is the output label of a bi-infinite path of input label $\bar0\omega_1\omega_2\omega_3\dots$, one has for any $i<i'$ in $\mathbb Z$
\begin{equation}\label{i'}
\frac{q_i}{\beta^i}=\sum_{i<j\le i'}\frac{\omega_j-\varepsilon_j}{\beta^j}+\frac{q_{i'}}{\beta^{i'}}
\end{equation}
which implies $\displaystyle\sum_{i\in\mathbb N}\frac{\omega_i}{\beta^i}=\sum_{i\in\mathbb Z}\frac{\varepsilon_i}{\beta^i}$ because $\displaystyle\lim_{i'\to+\infty}\frac{q_{i'}}{\beta^{i'}}=0$ (obvious) and $\displaystyle\lim_{i\to-\infty}\frac{q_i}{\beta^i}=0$ (for $q_{i-1}<q_i$ when $\omega_i=0$).
\hfill\end{proof}

\subsection{The number of redundant representations}For any finite sequence $\varepsilon_1\dots\varepsilon_k\in\{0,1,\dots,d-1\}^k$ we denote by $\mathcal N(\varepsilon_1\dots\varepsilon_k)$ the number of $\omega_1\dots\omega_k\in\{0,1,\dots,d-1\}^k$ such that
\begin{equation}\label{rep}
\omega_1\beta^{k-1}+\dots+\omega_k\beta^0=\varepsilon_1\beta^{k-1}+\dots+\varepsilon_k\beta^0.
\end{equation}
Notice that, this time, the carries $q_i=\sum_{j=i+1}^k\frac{\omega_j-\varepsilon_j}{\beta^{j-i}}$ do not belong to $(-1,\alpha]$ but to $(-\alpha,\alpha)$. The transducer $\mathcal T'$ we consider has the same arrows as $\mathcal T$, but its set of states $\{\mathfrak i_0=0,\mathfrak i_1,\dots,\mathfrak i_a\}$ is the smallest set containing $0$ and containing $\beta q-\omega+\varepsilon$ for any $q\in\{\mathfrak i_0=0,\dots,\mathfrak i_a\}$, $\omega\in\{0,1,\dots,d-1\}$, $\varepsilon\in\{0,1,\dots,\lceil\beta\rceil-1\}$ such that $\beta q-\omega+\varepsilon\in(-\alpha,\alpha)$.

\begin{pro}
$\mathcal N(\varepsilon_1\dots\varepsilon_k)$ is the number of paths in $\mathcal T'$, from $0$ to~$0$, with output label $\varepsilon_1\dots\varepsilon_k$. Equivalently,
$$
\mathcal N(\varepsilon_1\dots\varepsilon_k)=\begin{pmatrix}1&0&\dots&0\end{pmatrix}N_{\varepsilon_1}\dots N_{\varepsilon_k}\begin{pmatrix}1\\0\\\vdots\\0\end{pmatrix}
$$
where the $(0,1)$-matrices $N_\ell=(n^{\ell}_{ij})_{0\le i\le a\atop0\le j\le a}$, $\ell\in\{0,1,\dots,d-1\}$ are defined by
$$
n^{\ell}_{ij}=1\Leftrightarrow\beta\mathfrak i_i-\mathfrak i_j+\ell\in\{0,1,\dots,d-1\}.
$$
\end{pro}

\begin{proof}
There exists a path from $0$ to~$0$ with output label $\varepsilon_1\dots\varepsilon_k$, if and only if (\ref{Eucl-beta}) holds for some $q_0=0,q_1,\dots,q_k=0$ in $\{\mathfrak i_0,\dots,\mathfrak i_a\}$ and $\omega_1,\dots,\omega_k$ in $\{0,1,\dots,d-1\}$. This is equivalent to $\displaystyle\sum_{0<j\le k}\frac{\omega_j-\varepsilon_j}{\beta^j}=0$, because (\ref{Eucl-beta}) is equivalent to $\displaystyle\frac{q_{i-1}}{\beta^{i-1}}=\frac{\omega_i-\varepsilon_i}{\beta^i}+\frac{q_i}{\beta^i}$.
\hfill\end{proof}

\begin{exmp}\label{exGolden}If $\beta^2=3\beta-1$ and $d=3$, the transducer $\mathcal T'$ is

\includegraphics[trim=0 -100 0 -50,clip,scale=0.42]{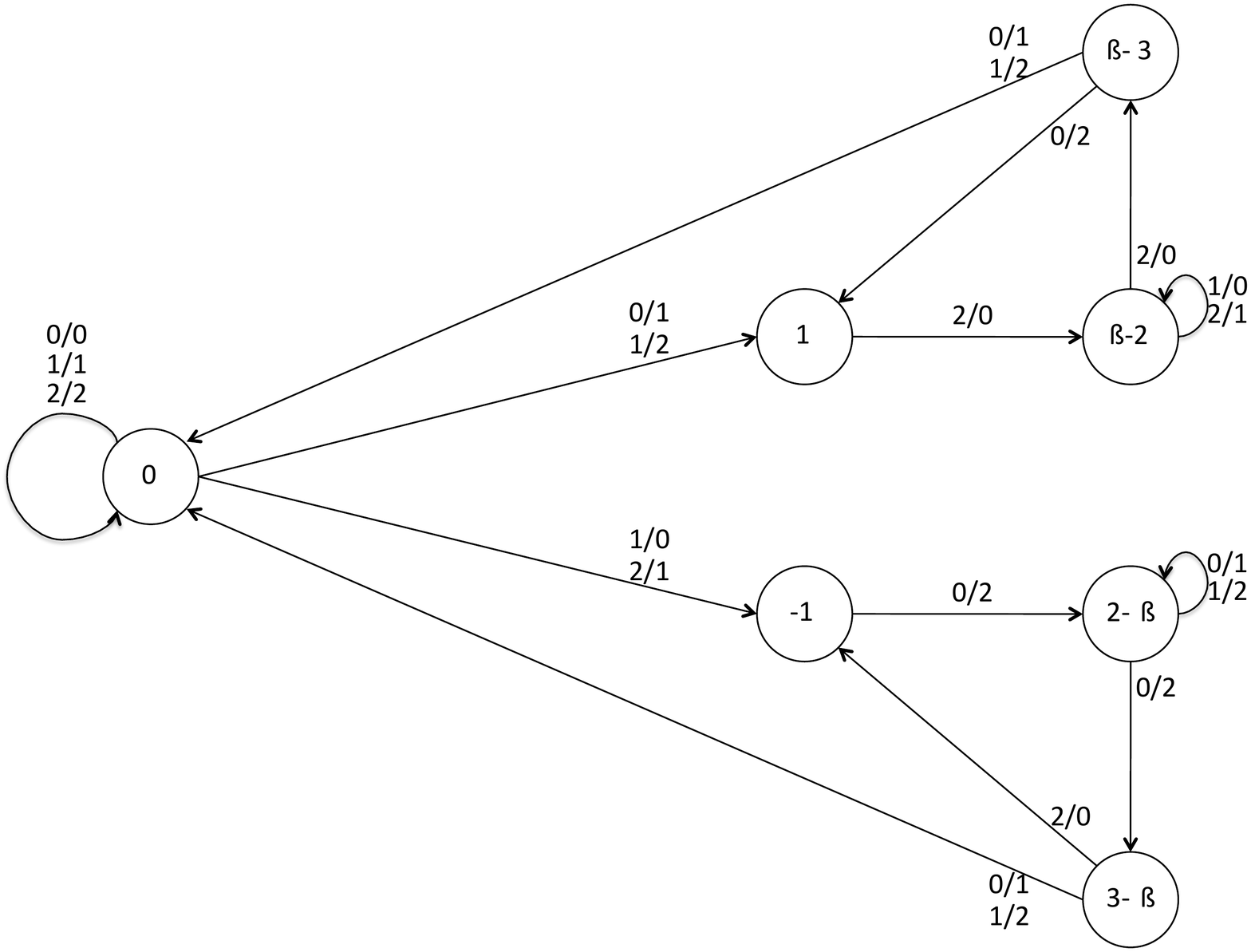}

\end{exmp}

\subsection{The linear representation of the Bernoulli convolutions in Pisot base}

The Bernoulli convolution \cite{Erd39,PSS} in real base $\beta>1$, associated to a positive probability vector $p=(p_0,\dots,p_{d-1})$, $d\ge\beta$, is the probability measure $\eta=\eta_{\beta,p}$ defined by setting, for any interval~\hbox{$I\subset\mathbb R$}
\begin{equation}
\eta_{\beta,p}(I):=P_p\Big(\Big\{(\omega_k)_{k\in\mathbb N}\;:\;0\le\omega_k\le d-1,\ \sum_k\frac{\omega_k}{\beta^k}\in I\Big\}\Big)
\end{equation}
where $P_p$ the product probability defined on $\{0,\dots,d-1\}^{\mathbb N}$ from the probability vector $p$.

In order to compute the value of $\eta_{\beta,p}$ on some suitable intervals, we recall the Parry condition of $\beta$-admissibility \cite[Theorem 2.3.11]{FS}:

\begin{defn}A sequence $(\varepsilon_k)_{k\in\mathbb N}$ is $\beta$-admissible, i.e. is the Rényi expansion of a real $x\in[0,1)$, iff $\forall k,\ \varepsilon_k\varepsilon_{k+1}\dots\prec\alpha_1\alpha_2\dots$ for the lexicographical order, where $(\alpha_k)_{k\in\mathbb N}$ is the quasi-expansion of the unity that can be defined by
$$
1=\sum_{i\in\mathbb N}\frac{\alpha_i}{\beta^i}\quad\hbox{and}\quad\forall i\ge0,\ 0<\sum_{j>i}\frac{\alpha_j}{\beta^j}\le\frac1{\beta^i}.
$$
\end{defn}
From now we assume that $\beta$ has a finite Rényi expansion \cite{R}, or equivalently that $(\alpha_k)_{k\in\mathbb N}$ has a period $T$.

\begin{lem}\label{classical}(i) A sequence is in the closure of Adm$_\beta$ (set of the $\beta$-admissible sequences), if and only if it is a infinite concatenation of words belonging to the set $\mathcal W$ defined by
$$
\begin{array}{rl}w\in\mathcal W\Leftrightarrow&w=\alpha_1\dots\alpha_T\hbox{ or}\\&\exists i\in\{1,\dots,T\},\alpha'_i\in\{0,\dots,\alpha_i-1\},\ w=\alpha_1\dots\alpha_{i-1}\alpha'_i.\end{array}
$$

(ii) If $\varepsilon_1\dots\varepsilon_k$ is a concatenation of words of $\mathcal W$, the $\beta$-adic interval  $I_{\varepsilon_1\dots\varepsilon_k}$, i.e. the set of the $x\in[0,1)$ whose Rényi expansion begins by $\varepsilon_1\dots\varepsilon_k$, is simply
$$
I_{\varepsilon_1\dots\varepsilon_k}=\Big[\sum_{i=1}^k\frac{\varepsilon_i}{\beta^i},\ \sum_{i=1}^k\frac{\varepsilon_i}{\beta^i}+\frac1{\beta^k}\Big).
$$
\end{lem}

\begin{proof}
(i) If $\varepsilon_1\varepsilon_2\dots\in\overline{\hbox{Adm}}_\beta$ one has $\varepsilon_1\dots\varepsilon_T\preceq\alpha_1\dots\alpha_T$, so there exists $k\le T$ such that
$$
\varepsilon_1\dots\varepsilon_k\in\mathcal W\hbox{ and }\varepsilon_{k+1}\varepsilon_{k+2}\dots\in\overline{\hbox{Adm}}_\beta.
$$
By iteration, $\varepsilon_1\varepsilon_2\dots=w_1w_2\dots$ with $\forall i,\ w_i\in\mathcal W$.

Conversely suppose that $\varepsilon_1\varepsilon_2\dots=w_1w_2\dots$ with $\forall i,\ w_i\in\mathcal W$. Then for any~$i$ there exists $j$ and a nonempty suffix $w'_j$ of $w_j$ such that
$$
\varepsilon_{i+1}\varepsilon_{i+2}\dots=w'_jw_{j+1}\dots.
$$
Clearly, from the definition of $\mathcal W$ one has $w_jw_{j+1}\dots\preceq\alpha_1\alpha_2\dots$. But there exists $k$ such that $w_j=\alpha_1\dots\alpha_kw'_j$, so one deduce
$$
w'_jw_{j+1}\dots\preceq\alpha_{k+1}\alpha_{k+2}\dots\preceq\alpha_1\alpha_2\dots,
$$
proving that $\varepsilon_1\varepsilon_2\dots\in\overline{\hbox{Adm}}_\beta$.

(ii) If $\varepsilon_1\dots\varepsilon_k$ is a concatenation of words of~$\mathcal W$, from (i) the eventually periodic sequence $\varepsilon_1\dots\varepsilon_k\overline{\alpha_1\dots\alpha_T}$ belongs to $\overline{\hbox{Adm}}_\beta$. Denoting by $\varepsilon_1\varepsilon_2\dots$ this sequence, the closed interval $\overline{I_{\varepsilon_1\dots\varepsilon_k}}$ contains $\displaystyle\sum_{i=1}^\infty\frac{\varepsilon_i}{\beta^i}=\sum_{i=1}^k\frac{\varepsilon_i}{\beta^i}+\frac1{\beta^k}$.
\hfill\end{proof}


In order to define the matrices associated to $\eta_{\beta,p}$, we consider a set of reals that may be sligthly different from the set $S'_{\beta,d}$ defined in Subsection~\ref{thetr}: let $\{\mathfrak j_0=0,\dots,\mathfrak j_{a'}\}$ be the smallest set containing $0$ and containing $\beta q-\omega+\varepsilon$, whenever the three reals $q\in\{\mathfrak j_0,\dots,\mathfrak j_{a'}\}$, $\omega\in\{0,1,\dots,d-1\}$, $\varepsilon\in\{0,1,\dots,\lceil\beta\rceil-1\}$ satisfy the condition $\beta q-\omega+\varepsilon\in(-1,\alpha)$. The matrices $M_\ell=(m^{\ell}_{ij})_{0\le i\le a'\atop0\le j\le a'}$, $\ell\in\{0,1,\dots,\lceil\beta\rceil-1\}$, are defined by
$$
m^{\ell}_{ij}=\left\{\begin{array}{ll}0&\hbox{if }\beta\mathfrak j_i-\mathfrak j_j+\ell\not\in\{0,1,\dots,d-1\}\\p_{\beta\mathfrak j_i-\mathfrak j_j+\ell}&\hbox{else.}\end{array}\right.
$$

\begin{thm}\label{linrepPisotcase}
If $\varepsilon_1\dots\varepsilon_k$ is a concatenation of words of $\mathcal W$, one has for any $i\in\{0,\dots,a'\}$
\begin{equation}\label{valuesPisotcase}
\eta(\mathfrak j_i+I_{\varepsilon_1\dots\varepsilon_k})=E_iM_{\varepsilon_1}\dots M_{\varepsilon_k}C
\end{equation}
where $E_0,E_1,\dots,E_{a'}$ are the canonical basis $(a'+1)$-dimensional row vectors and $C$ a suitable positive eigenvector of the irreducible matrix $\sum_{w\in\mathcal W}M_w$. Consequently one can define a sofic measure on the symbolic space $\mathcal W^{\mathbb N}$, by setting
\begin{equation}
\eta_i[w_1\dots w_k]:=\frac{\eta(\mathfrak j_i+I_{w_1\dots w_k})}{\eta(\mathfrak j_i+[0,1))}.
\end{equation}
\end{thm}

\begin{proof}Let $I_j=\big[\sum_{i=j+1}^k\frac{\varepsilon_i}{\beta^{i-j}},\ \sum_{i=j+1}^k\frac{\varepsilon_i}{\beta^{i-j}}+\frac1{\beta^{k-j}}\big)$  for any $0\le j\le k$; so one has $I_0=I_{\varepsilon_1\dots\varepsilon_k}$ (by Lemma \ref{classical} (ii)), $I_k=[0,1)$ and for any $j\ne0$
$$
\beta I_{j-1}=\varepsilon_j+I_j.
$$
Assuming by convention that $p_i=0$ for any $i\not\in\{0,1,\dots,d-1\}$,
$$
\begin{array}{rcl}\eta(\mathfrak j_i+I_{j-1})&=&\displaystyle\sum_{\iota=0}^{d-1}P\Big(\big\{\omega_1=\iota\hbox{ and }\sum_{k=1}^\infty\frac{\omega_{k+1}}{\beta^k}\in\beta\mathfrak j_i+\beta I_{j-1}-\iota\big\}\Big)\\&=&\displaystyle\sum_{\iota=0}^{d-1}p_\iota P\Big(\big\{\sum_{k=1}^\infty\frac{\omega_{k+1}}{\beta^k}\in\beta\mathfrak j_i-\iota+\varepsilon_j+I_j\big\}\Big)\\&=&\displaystyle\sum_{i'=0}^{a'}p_{\beta\mathfrak j_i-\mathfrak j_{i'}+\varepsilon_j}\ \eta\left(\mathfrak j_{i'}+I_j\right)\end{array}
$$
so one has
$$
\begin{pmatrix}\eta(\mathfrak j_0+I_{j-1})\\\vdots\\\eta(\mathfrak j_{a'}+I_{j-1})\end{pmatrix}=M_{\varepsilon_j}\begin{pmatrix}\eta(\mathfrak j_0+I_j)\\\vdots\\\eta(\mathfrak j_{a'}+I_j)\end{pmatrix}$$
and (\ref{valuesPisotcase}) follows by induction, with
\begin{equation}\label{firstdef}
C:=\begin{pmatrix}\eta([\mathfrak j_0,\mathfrak j_0+1))\\\vdots\\\eta([\mathfrak j_{a'},\mathfrak j_{a'}+1))\end{pmatrix}.
\end{equation}

From Lemma \ref{classical} (i) the intervals $I_w$, $w\in\mathcal W$, form a partition of $[0,1)$. Denoting by $M_w$ the product matrix $M_{\varepsilon_1}\dots M_{\varepsilon_k}$ for any $w=\varepsilon_1\dots\varepsilon_k$ one deduce that the column vector $C$, defined as in (\ref{firstdef}), satifies
$$
\sum_{w\in\mathcal W}M_wC=C.
$$
The entries of $C$ are positive because the intervals $[\mathfrak j_i,\mathfrak j_i+1)$ intersect the interior of $[0,\alpha)$ (the support of $\eta$). Now $\sum_{w\in\mathcal W}M_w$ is irreducible because, from the definition of $\{\mathfrak j_0,\dots,\mathfrak j_{a'}\}$, there exists a path from any $\mathfrak j_i$ to any $\mathfrak j_{i'}$ in the transducer $\mathcal T_\eta$ of set of states $\{\mathfrak j_0,\dots,\mathfrak j_{a'}\}$ and arrows $q{{\omega\ /\ \varepsilon\atop\xrightarrow{\hspace*{50pt}}}\atop\phantom{}}\beta q-\omega+\varepsilon$.
\hfill\end{proof}

\begin{exmp}\label{exgold}This is the transducer $\mathcal T_\eta$ in the case $\beta=\frac{1+\sqrt5}2$ and $d=2$, the set of states being $\{\mathfrak j_0=0,\mathfrak j_1=1,\mathfrak j_2=\beta-1\}$:

\includegraphics[trim=0 220 0 170,clip,scale=0.42]{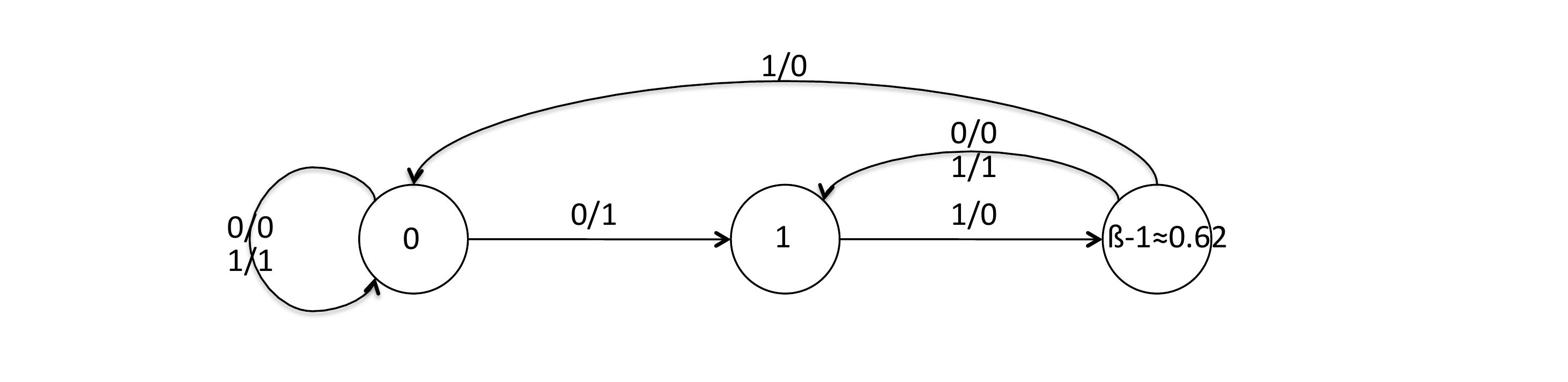}

For $\ell\in\{0,1\}$ and $i,j\in\{0,1,2\}$, the $(i,j)$-entry of $M_\ell$ is $p_\omega$ if $\omega=\beta\mathfrak j_i-\mathfrak j_j+\ell$ is an integer, i.e. when there exist a arrow $\mathfrak j_i{{\omega\ /\ \ell\atop\xrightarrow{\hspace*{50pt}}}\atop\phantom{}}\mathfrak j_j$:
$$
M_0=\begin{pmatrix}p_0&0&0\\0&0&p_1\\p_1&p_0&0\end{pmatrix}\quad\hbox{and}\quad M_1=\begin{pmatrix}p_1&p_0&0\\0&0&0\\0&p_1&0\end{pmatrix}.
$$

We have $\mathcal W=\{0,10\}$. So we can construct from $\mathcal T_\eta$ a new transducer with output labels $0$ and $10$:

\includegraphics[trim=0 170 0 150,clip,scale=0.42]{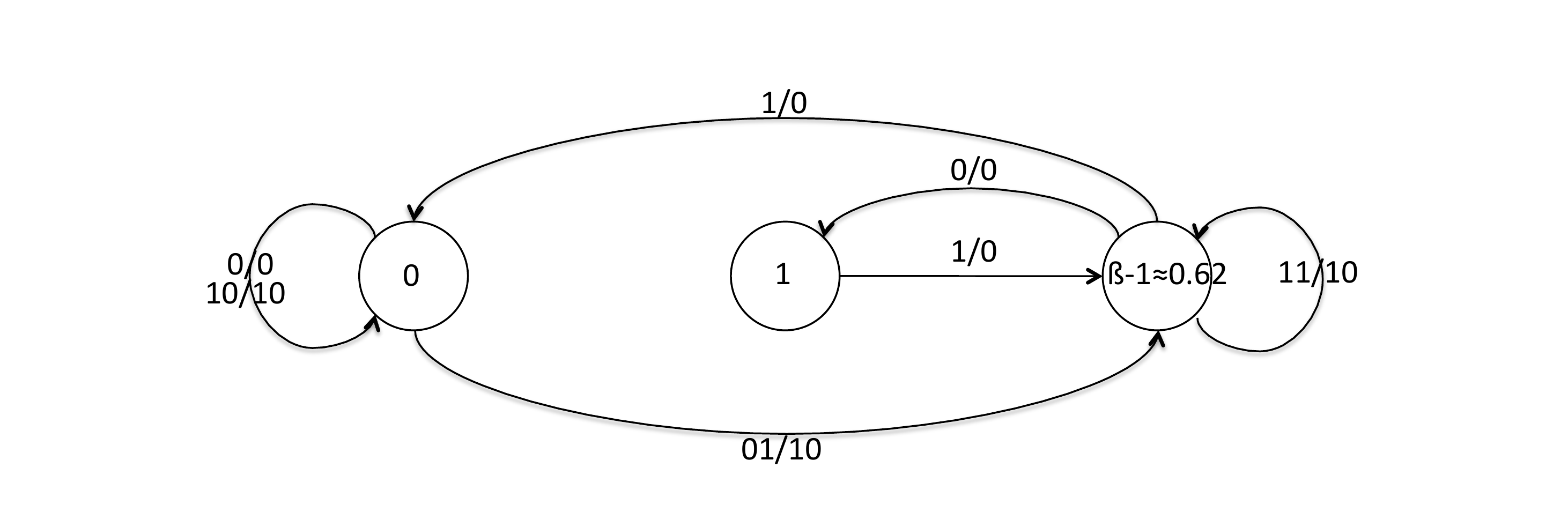}

We deduce the graph of the Markov chain associated to each of the the sofic measures $\eta_0,\eta_1,\eta_2$:

\includegraphics[trim=0 50 0 50,clip,scale=0.42]{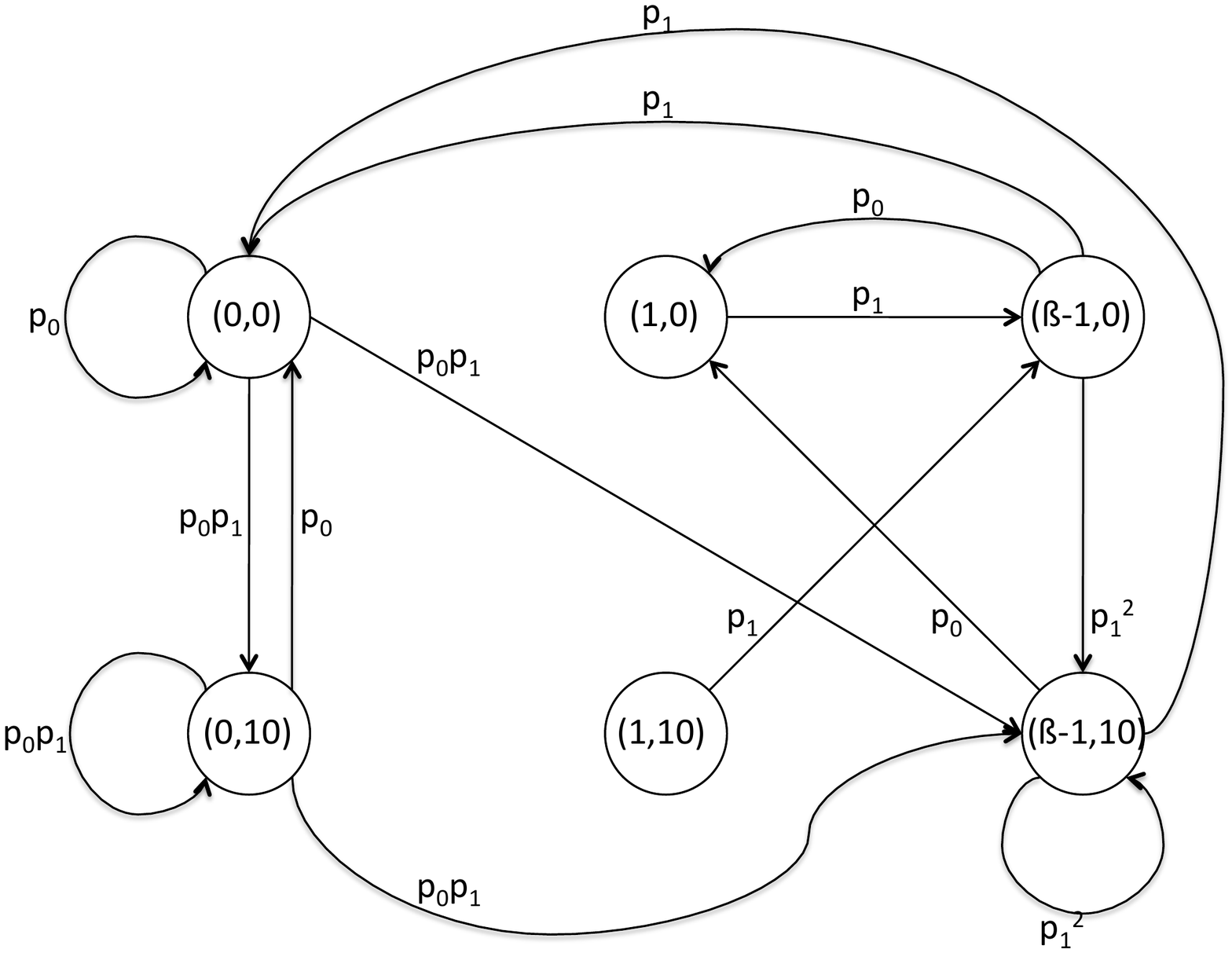}

and its transition matrix:
$$
\begin{pmatrix}M_0&M_{10}\\M_0&M_{10}\end{pmatrix}=\begin{pmatrix}p_0&0&0&p_0p_1&0&p_0p_1\\0&0&p_1&0&0&0\\p_1&p_0&0&0&0&{p_1}^2\\p_0&0&0&p_0p_1&0&p_0p_1\\0&0&p_1&0&0&0\\p_1&p_0&0&0&0&{p_1}^2\end{pmatrix}.
$$
\end{exmp}

\section{More about the representations in base $b=2$ with three digits}

Let $\mathcal N(n)$ be the number of representations, with digits in $\{0,1,2\}$, of the integer
$$
n=_{_2}1^{a_s}0^{a_{s-1}}\dots0^{a_1}1^{a_0},
$$
where $a_0\ge0$ ($a_0=0$ iff $n$ is even), and $a_1,\dots,a_s\ge1$. According to \cite[Proposition 6.11]{D}, $\mathcal N(n)$ is the denominator $q_s$ of the continued fraction $\displaystyle[0;a_1,\dots,a_s]:=\frac1
{\displaystyle a_1+\frac1
{\displaystyle a_2+\dots+{\atop\displaystyle\frac1
{a_s}}}}$. From \cite[Theorem 19]{FLT} there exists $\alpha_0>0$ and $S\subset\mathbb N$ of natural density~$1$ such that
\begin{equation}\label{alpha0}
\lim_{n\in S,\ n\to\infty}\frac{\log \mathcal N(n)}{\log n}=\alpha_0.
\end{equation}
From the classical formula giving the denominator of the continued fraction in terms of $a_1,\dots,a_s$, one has
$$
\alpha_0=\lim_{k\to\infty}\frac1{2^k\log(2^k)}\sum_{s,a_1,\dots,a_s}\log\left(\sum_{h,i_0,\dots,i_{h+1}}\prod_{j=0}^ha_{i_j,i_{j+1}}\right),
$$
summing for $s,a_1,\dots,a_s\in\mathbb N$ such that $\sum a_i\le k$ and for $h\in\mathbb N\cup\{0\}$, $(i_0,\dots,i_{h+1})\in\{0\}\times\mathbb N^h\times\{s+1\}$, where
$$
a_{i',i}:=\left\{\begin{array}{ll}a_i&\hbox{if }i-i'\hbox{ is an odd positive integer}\\0&\hbox{else.}\end{array}\right.
$$

By the following computation, using the properties of the function $\mathcal N$, one obtains $\alpha_0\approx0.56$:

$f := $proc$ (n):\ $if$\ n=2*$floor$(n/2)\ $then$\ f(n/2)+f(n/2-1)\ $else$\ f(n/2-1/2)\ $end$\ $if$\ $end$\ $proc$:$

$f(0) := 1:$
 
$s := $seq$($ln$(f(n)),\ n = 0 .. 16383):$

$\alpha_0:= $evalf$($sum$(s[n],\ n = 8192 .. 16383)/(8192*$ln$(8192)));$

\begin{pro}\label{ae}The constant $\alpha_0$ defined in (\ref{alpha0}) is also the Lebesgue-a.e. value of $\displaystyle\lim_{s\to\infty}\frac{\log q_s(x(t))}{s\log4}$, where $x(t)=[0;a_1,a_2,\dots]$ for any $t=_{_2}0.1^{a_1}0^{a_2}1^{a_3}\dots$ in the interval $\big[\frac12,1\big)$.
\end{pro}

\begin{proof}From \cite[Remark 21]{FLT}, $\alpha_0\log2$ is the constant $\gamma$ defined in \cite[Theorem 1.1]{SF}. Now by \cite[Corollary 1.2]{SF}, $\frac{\log3}{\log2}-\alpha_0$ is the almost sure value of the local dimension $d_{\rm loc}(t)$ of the Bernoulli convolution $\eta=\eta_{2,3}$ (in base $2$ with three digits).

In the sequel we denote by $0.1^{a_1}0^{a_2}1^{a_3}\dots$ or by $\varepsilon_1\varepsilon_2\dots$ the expansion of the real $t\in\big[\frac12,1\big)$ in base~$2$ (we suppose it is not eventually $0$ nor eventually $1$), and for any $k\ge1$ we put
$$
n(t,k):=\varepsilon_12^{k-1}+\dots+\varepsilon_k2^0.
$$
In case $k=a_1+\dots+a_s$ for some even integer~$s$, the formula of \cite[Proposition 6.11]{D} gives $\mathcal N(n(t,k))=q_s(x(t))$. By (\ref{etabpdeI}),
\begin{equation}\label{samesame}
\mathcal N(n(t,k))+\mathcal N(n(t,k)-1)=2\cdot3^k\eta(I_{\varepsilon_1\dots\varepsilon_k}).
\end{equation}
But for any integer $n$, $\mathcal N(n-1)$ is in a sense close to $\mathcal N(n)$: by \cite[Lemma~9~(ii)]{FLT} one has $\displaystyle\frac1{C\log n}\le\frac{\mathcal N(n)}{\mathcal N(n-1)}\le C\log n$. So (\ref{samesame}) implies
\begin{equation}\label{samesamesame}
\lim_{s\ {\rm even}\atop s\to\infty}\frac{\log q_s(x(t))}{a_1+\dots+a_s}=\lim_{k\to\infty}\frac{\log\left(2\cdot3^k\eta(I_{\varepsilon_1\dots\varepsilon_k})\right)}k=\log3-d_{\rm loc}(t)\log2.
\end{equation}
Since, from the law of large numbers, $a_1+\dots+a_s$ is equivalent to~$2s$ for Lebesgue-a.e. $t$, (\ref{samesamesame}) implies for Lebesgue-a.e. $t\in\big[\frac12,1\big)$
$$
\lim_{s\to\infty}\frac{\log q_s(x(t))}{2s\log2}=\frac{\log3}{\log2}-d_{\rm loc}(t)=\alpha_0.
$$
\hfill\end{proof}

\begin{rem}
Using the matricial expression of the denominators of the continued fractions, Proposition \ref{ae} means that $\alpha_0\log2$ is the Lyapunov exponent of the set of matrices $\Big\{\begin{pmatrix}1&0\\1&1\end{pmatrix},\begin{pmatrix}1&1\\0&1\end{pmatrix}\Big\}$.
\end{rem}

\begin{rem}
The Levy constant
\begin{equation}\label{Levy}
\lim_{s\to\infty}\frac{\log q_s(x)}s=\frac{\pi^2}{12\log2}\approx1.18657\quad\hbox{for Lebesgue-a.e. }x\in(0,1)
\end{equation}
is larger than the value of
\begin{equation}\label{alphazero}
\lim_{s\to\infty}\frac{\log q_s(x(t))}s=\alpha_0\log4\approx0.78\quad\hbox{for Lebesgue-a.e. }t\in\Big[\frac12,1\Big).
\end{equation}
On the other side it is well known that the partial quotients $a_i(x)$ of Lebesgue-a.e. real $x$ satisfy
\begin{equation}\label{infty}
\lim_{s\to\infty}\frac{a_1(x)+\dots a_s(x)}s=\infty
\end{equation}
For any $t$ such that $x(t)$ satisfies (\ref{Levy}) and (\ref{infty}), the l.h.s. in (\ref{samesamesame}) tends to $0$, consequently the local dimension at~$t$ is maximal: $d_{\rm loc}(t)=\frac{log3}{\log2}\approx1.585$. Since $\frac{log3}{\log2}-\alpha_0\approx1.025$ and $\frac{log3}{\log2}-\frac{log\frac{1+\sqrt5}2}{\log2}\approx0.891$ (the minimal local dimension because $q_s\le\big(\frac{1+\sqrt5}2\big)^{a_1+\dots+a_s}$ by induction), the graph of the singularity spectrum of $\eta_{2,3}$ has the following form:

\includegraphics[trim=0 112 0 0,clip,scale=0.42]{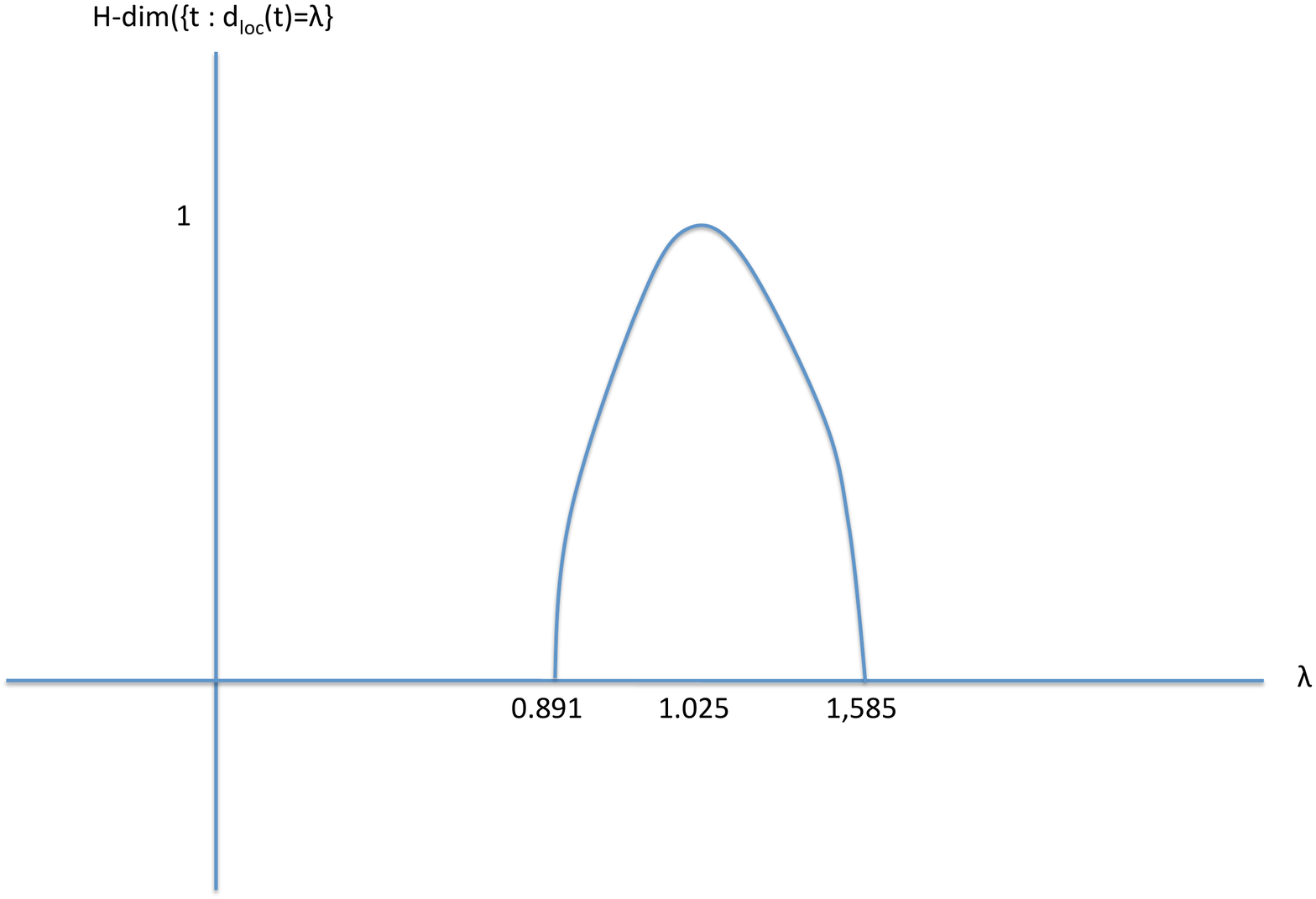}
\end{rem}

\end{document}